\documentclass[a4paper]{amsart}
\usepackage{amsfonts}\usepackage{amssymb}\usepackage{amsmath}\usepackage{amsthm}
\usepackage{latexsym}
\usepackage{graphicx}
\usepackage{layout}
\usepackage{enumerate}
\usepackage{subfigure}
\DeclareMathOperator\arctanh{arctanh}
\usepackage{caption}
\usepackage{pgfplots}
\usepackage{appendix}
\usepackage[colorlinks=true,urlcolor=blue,citecolor=red,linkcolor=blue,linktocpage,pdfpagelabels,bookmarksnumbered,bookmarksopen]{hyperref}
\usepackage{tikz}
\usepackage{tikz-3dplot}
\usetikzlibrary{hobby}
\usepackage{ulem}
\newtheorem{theorem}{Theorem}[section]

\newtheorem{proposition}[theorem]{Proposition}
\newtheorem{lemma}[theorem]{Lemma}
\newtheorem{corollary}[theorem]{Corollary}
\newtheorem{remark}[theorem]{Remark}
\newtheorem{thm}{Theorem}[section]

\theoremstyle{definition}
\newtheorem{defn}[thm]{Definition}
\theoremstyle{remark}
\newtheorem{rem}[thm]{Remark}

\numberwithin{equation}{section}
\usepackage{graphicx}

\renewcommand{\kappa}{{k_\Omega}}
\def\R{\mathbb{R}}

\def\di12{\mathcal{D}^{1,2}(\R^n)}

\def\l{{\lambda}}

\def\0l{_{0,\l}}
\def\1l{_{1,\l}}
\def\2l{_{2,\l}}
\def\3l{_{3,\l}}
\def\4l{_{4,\l}}

\def\Om{\Omega}

\def\beq{\begin{equation}}
\def\eeq{\end{equation}}
\def\sideremark#1{\ifvmode\leavevmode\fi\vadjust{\vbox to0pt{\vss
 \hbox to 0pt{\hskip\hsize\hskip1em
 \vbox{\hsize2.1cm\tiny\raggedright\pretolerance10000
  \noindent #1\hfill}\hss}\vbox to15pt{\vfil}\vss}}}%

\newtheorem*{theorem*}{Theorem}

\begin{document}
\title[]{ The role of the curvature of a surface in the shape of the solutions to elliptic equations}
\thanks{
The second author
acknowledges support of INdAM - GNAMPA.  \\
The third author acknowledges support of the project ``Perturbation problems and asymptotics for elliptic differential equations: variational and potential theoretic methods'' funded by the European Union – Next Generation EU and by MUR-PRIN-2022SENJZ3 and from the project ``Analisi Geometrica e Teoria Spettrale su varietà Riemanniane ed Hermitiane'' of the INdAM GNSAGA.\\
The first author acknowledges support of INdAM - GNAMPA Project,  CUP E53C25002010001.  
The first author acknowledges support of the project e.INS- Ecosystem of Innovation for
Next Generation Sardinia (cod. ECS 00000038) funded by the Italian Ministry for Research and Education
(MUR) under the National Recovery and Resilience Plan (NRRP) - MISSION 4 COMPONENT 2, ``From
research to business'' INVESTMENT 1.5, ``Creation and strengthening of Ecosystems of innovation'' and
construction of ``Territorial R\&D Leaders'' and of the project Next Generation EU-CUP-J55F21004240001,  DM 737-2021, risorse $2022-2023$. }
\author[Gladiali]{Francesca Gladiali}
\address{Dipartimento di Scienze CFMN,  Universit\`a degli Studi di Sassari, Via Vienna 2 - 07100 Sassari, Italy, e-mail: {\sf fgladiali@uniss.it}.}
\author[Grossi]{Massimo Grossi }
\address{Dipartimento di Scienze di Base e Applicate per l'Ingegneria, Universit\`a di Roma ``La Sapienza'', Via Scarpa 12 - 00161 Roma, Italy, e-mail: {\sf massimo.grossi@uniroma1.it}.}
\author[Provenzano]{Luigi Provenzano}
\address{Dipartimento di Scienze di Base e Applicate per l'Ingegneria, Universit\`a di Roma ``La Sapienza'', Via Scarpa 12 - 00161 Roma, Italy, e-mail: {\sf luigi.provenzano@uniroma1.it}.}

\begin{abstract}
We prove uniqueness and non-degeneracy of the critical point of positive, semi-stable solutions of $-\Delta u=f(u)$ with Dirichlet boundary conditions for a class of star-shaped domains of the sphere and of the hyperbolic plane satisfying a geometric condition.  In the spherical case, this condition is weaker than convexity, while in the hyperbolic case it is weaker than horoconvexity. Finally, we construct examples showing that this geometric condition is indeed optimal.
\end{abstract}


\maketitle

\section{Introduction and statement of the main result}
Let $\Omega$ be a smooth domain of $M$, where $M=\mathbb S^2,\mathbb H^2$ with the standard metrics of constant curvature $1,-1$,  and let $u$ be a solution to the following semilinear problem:
\begin{equation}\label{eq1}
\begin{cases}
-\Delta u=f(u)\,, & {\rm in\ }\Omega,\\
u=0\,, & {\rm on\ }\partial\Omega.
\end{cases}
\end{equation}
Here $f:[0,+\infty)\to\mathbb R$ is a $C^1$ nonlinearity satisfying $f(0)\geq 0$. We consider {\it positive} solutions of \eqref{eq1} that are {\it semi-stable}, meaning that the (linear) Schr\"odinger operator
\begin{equation}\label{semistable}
-\Delta-f'(u)I
\end{equation}
is non-negative in $\Omega$.

Two classical problems that fall in this setting are the {\it torsion problem}, i.e.,
\begin{equation}\label{torsion}
f(s)=1
\end{equation}
and the {\it Dirichlet eigenvalue problem} (when we consider the {\it first eigenfunction}), i.e.,
\begin{equation}\label{eigenvalue}
f(s)=\lambda_1 s,
\end{equation}
where $\lambda_1$ is the first Dirichlet eigenvalue of $\Omega$.

\subsection{Discussion}
The study of critical points of solutions of elliptic equations is a vast topic. Typical questions involve estimates on the number of critical points and their location. Some  introductory surveys on the subject are \cite{Gera,magnanini}. The majority of available results concern the torsion problem \eqref{torsion} and the eigenvalue problem \eqref{eigenvalue}, especially for convex domains in $\mathbb R^n$. For a general non-linearity $f$, very little is known outside the class of convex planar domains.

\subsubsection{The first Dirichlet eigenfunction} This is perhaps the most studied case. The uniqueness of the critical point in the case of convex Euclidean domains is well-known \cite{AcPaPh,BrLi,CaFr}. This result (as most of the results for problems \eqref{torsion} and \eqref{eigenvalue}) is obtained in a somehow indirect way, namely it is a consequence of the fact that the first eigenfunction is log-concave, joined with the interior regularity properties of the eigenfunction. 

For convex domains of the sphere $\mathbb S^n$, the analogous result has been obtained in \cite{LeeWa}, again as a consequence of log-concavity estimates on the eigenfunction (see also \cite{DSW,HWZ}). In the case of convex domains of $\mathbb S^2$ with diameter smaller than $\frac{\pi}{2}$, a direct proof of the uniqueness and non-degeneracy of the critical point has been presented in \cite{GP1}. Log-concavity of the eigenfunction for convex domains in surfaces of positive curvature has been studied in \cite{Khan1}.  A new point of view has been taken in \cite{Khan2}, where log-concavity estimates are investigated when a given metric undergoes a conformal deformation.

Concerning the hyperbolic space $\mathbb H^n$, in dimension $n=2$ it has been proved in \cite{GP1} that the first eigenfunction on {\it horoconvex} domains has exactly one, non-degenerate critical point. The result seems to be optimal in the sense that there exist examples of convex domains (which are not horoconvex) for which the first eigenfunction has at least two critical points \cite{shih}, or for which the level sets are non-convex \cite{BCNHS}. A log-concavity estimate has been recently proved \cite{Khan2}  in the case of horoconvex domains of $\mathbb H^2$ with sufficiently small diameter.

If we leave the realm of convex domains, estimates on the number of critical points for the first Dirichlet eigenfunction become more difficult or even unavailable. In the case of manifolds, we mention \cite{EPS15,EPS17} where the authors are able to find metrics on a given Riemannian manifold for which the first eigenfunction has an arbitrary number of critical points.

\subsubsection{The torsion function} Surprisingly enough, there are fewer results on the critical points of the torsion functions compared to the first eigenfunction. A seminal  paper is \cite{MaLi} where the author proves the convexity of the level lines of the torsion function on planar convex domains, and the uniqueness and non-degeneracy of the critical point. On the other hand, if we drop the convexity assumption, we have no bounds on the number of critical points, even if a domain is close in a suitable sense to a convex one \cite{GG22}. The result of \cite{MaLi} has been extended to any dimension in \cite{KoLe}.

In \cite{GP1} it has been proved that the torsion function has a unique, non-degenerate critical point for horoconvex domains of $\mathbb H^2$ and for convex domains of the sphere of diameter smaller than $\frac{\pi}{2}$. In \cite{Khan2} the authors prove the quasi-concavity of the torsion function for convex spherical domains of small diameter (less than $2\arctan(1/5)$). A complete proof of the conjectured $\frac{1}{2}$-concavity of the torsion function in spherical geometry is currently unavailable. Up to our knowledge, no result is available in dimension $n\geq 3$.

In \cite{GG22} the authors exhibit a family of star-shaped planar domains for which the torsion function has an arbitrary number of critical points. These domains are ``not too far'' from being convex. 

\subsubsection{General non-linearity}
For a general non-linearity $f$, the results in the celebrated paper \cite{GNN} imply that for a domain in $\mathbb R^n$ which is symmetric with respect to a point and convex in any direction, the critical point of a positive solution of \eqref{eq1} is unique and non-degenerate. Again, the statement on the critical point is an indirect result which follows from the symmetry properties of a solution derived in \cite{GNN}.

A second relevant result is \cite{CC} where the authors confine to strictly convex planar domains, dropping symmetry conditions, and requiring positivity and semi-stability of solutions. Under these hypotheses, they prove uniqueness and non-degeneracy of the critical point. Strict convexity has been relaxed in \cite{DRG1}. We remark that in \cite{CC,DRG1} the study of the critical points has been attacked {\it directly}. Roughly speaking, the combinatorial properties of critical points of smooth functions on differentiable manifolds are related to the topology of the manifold though its {\it Euler characteristic}. However, classical theorems in differential topology (as the Poincaré-Hopf Theorem, see Theorem \ref{PH}) imply that the topological information on the manifolds gives an information on a {\it signed} sum of the critical points, the sign depending on the nature of the critical points. With no other information, the research for estimates on the number of critical points is a very difficult task. Moreover, the difficulties of taking a direct approach increase considerably in dimension $n\geq 3$. We refer to \cite{Gera,magnanini} for a deeper discussion on this point.

For domains of $\mathbb H^2$ and $\mathbb S^2$,  to our knowledge, almost nothing was known for general non-linearities until \cite{GP1}. Namely, in \cite{GP1} it has been proved that positive, semi-stable solutions of \eqref{eq1} have a unique-non degenerate critical point in the following cases: 
\begin{itemize}
\item[\textbullet]$\Omega$ is a horoconvex set of $\mathbb H^2$;
\item[\textbullet] $\Omega$ is a uniformly convex set (smooth with strictly positive geodesic curvature) of $\mathbb S^2$ with diameter less than $\frac{\pi}{2}$.
\end{itemize}
The results of \cite{GP1} are in the spirit of \cite{CC}, in the sense that they descend from an application of the Poincaré-Hopf Theorem to the counting of the index of a suitable vector field related to $\nabla u$.

The results of \cite{GP1} are quite good in the case of $\mathbb H^2$ in view of the examples in \cite{BCNHS,shih}, however the restriction on the diameter for domains in $\mathbb S^2$ is unsatisfactory. Roughly speaking, the negative curvature of $\mathbb H^2$ forces to require more than convexity, and horoconvexity seems to be a right condition. On the other hand, for $\mathbb S^2$ we expect that something weaker than convexity should be required for the uniqueness and non-degeneracy of the critical point. This is the type of statement that we prove in this paper. Moreover, surprisingly enough, we are also able to go beyond horoconvexity in the hyperbolic setting, and to  extend the result of \cite{GP1} to a class of domains larger than horoconvex domains.

\subsection{Statement of the main results}

In order to state the main results of the paper we recall the definition of uniformly star-shaped domain.
\begin{defn}\label{starS}
We say that a smooth domain $\Omega$ in $\mathbb S^2$ or $\mathbb H^2$ is {\it uniformly star-shaped} with respect to some point $P\in\Omega$ if 
\begin{itemize}
\item for any $p\in\partial\Omega$ the shortest geodesic from $P$ to $p$ is contained in $\Omega$;
\item there exists $\delta>0$ such that $\langle \nu,\vec{e_r}\rangle\geq\delta$ on $\partial\Omega$, where $\nu$ is the unit outer normal to $\partial\Omega$, and $\vec{e_r}$ is the coordinate vector field associated to the radial coordinate $r$ based at the point $P$.
\end{itemize}
In the case of $\mathbb S^2$ we also have $-P\notin\overline\Omega$.
\end{defn}

The first result concerns  {\it spherical domains}. In the spherical case it is customary to denote the radial coordinate with respect to a point $P$ with $\theta$ (instead of $r$), so we will adopt this notation as well.

\begin{thm}\label{main}
Let $\Omega$ be a smooth domain of $\mathbb S^2$, uniformly star-shaped with respect to $P\in\Omega$. Let $\nu$ be the unit outer normal to $\partial\Omega$ and let $\kappa$ be the geodesic curvature of $\partial\Omega$ with respect to $\nu$. Assume that
\begin{equation}\label{assumption}
\cos(\theta)\kappa+\sin(\theta)\langle\nu,\vec{e_{\theta}}\rangle>0\,\ \ \ {\rm\ on\ }\partial\Omega\,,
\end{equation}
where $\theta$ is the geodesic distance from $P$ and $\vec{e_{\theta}}$ is the corresponding coordinate vector field. Then any positive, semi-stable solution of \eqref{eq1} has a unique non-degenerate critical point, which is a maximum.

Finally \eqref{assumption} is sharp in the sense that, for any fixed $n\in\mathbb N$  there exists $b_0>0$ and a family $\{\Omega_b\}_{b\in(0,b_0)}$ of domains such that
\begin{equation}\label{assumptionmf}
\lim\limits_{b\to0}(\cos(\theta)k_{\Omega_b}+\sin(\theta)\langle\nu,\vec{e_{\theta}}\rangle)\ge0\,\ \ \hbox{ on }\partial\Omega_b
\end{equation}
and the solution $u_b$ of the problem
\begin{equation}
\begin{cases}
-\Delta u=1\,, & {\rm in\ }\Omega_b,\\
u=0\,, & {\rm on\ }\partial\Omega_b.
\end{cases}
\end{equation}
admits at least $n$ maxima for all $b\in(0,b_0.)$
\end{thm}

The geometric condition \eqref{assumption} involves both the geodesic curvature of the boundary and the constant $\delta$ in the definition of the uniform star-shapedness (point $ii)$ of Definition \ref{starS}); moreover, the condition \eqref{assumption} takes into account the possibility that a domain exceeds the equator. Note that all convex domains satisfy the hypothesis of Theorem \ref{main}. In Section \ref{examples} we provide examples of non-convex domains and domains exceeding the equator satisfying the assumptions of Theorem \ref{main}.

The second result concerns {\it hyperbolic domains}. It extends the result of \cite{GP1} to class of domains larger than horoconvex domains.

\begin{thm}\label{main2}
Let $\Omega$ be a bounded smooth domain of $\mathbb H^2$, uniformly star-shaped with respect to $P\in\Omega$. Let $\nu$ be the unit outer normal to $\partial\Omega$ and let $\kappa$ be the geodesic curvature of $\partial\Omega$ with respect to $\nu$. Assume that
\begin{equation}\label{assumption2}
\cosh(r)\kappa-\sinh(r)\langle\nu,\vec{e_{r}}\rangle>0\,\ \ \ {\rm\ on\ }\partial\Omega\,,
\end{equation}
where $r$ is the geodesic distance from $P$ and $\vec{e_r}$ is the corresponding coordinate vector field. Then any positive, semi-stable solution of \eqref{eq1} has a unique non-degenerate critical point, which is a maximum.  Finally \eqref{assumption2} is sharp in the sense that, for any fixed $n\in\mathbb N$  there exists $b_0>0$ and a family $\{\Omega_b\}_{b\in(0,b_0)}$ of domains such that
\begin{equation}\label{assumptionmf2}
\lim\limits_{b\to0}(\cosh(r)k_{\Omega_b}-\sinh(r)\langle\nu,\vec{e_r}\rangle)\ge0\,\ \ \hbox{ on }\partial\Omega_b
\end{equation}
and the solution $u_b$ of the problem
\begin{equation}
\begin{cases}
-\Delta u=1\,, & {\rm in\ }\Omega_b,\\
u=0\,, & {\rm on\ }\partial\Omega_b.
\end{cases}
\end{equation}
admits at least $n$ maxima for all $b\in(0,b_0)$.
\end{thm}
The geometric condition \eqref{assumption2}, as in the spherical case, involves both the geodesic curvature of the boundary and the constant $\delta$ in the definition of uniform starshapedness; we note that all smooth horoconvex domains satisfy \eqref{assumption2} (hence we recover the result in \cite{GP1}); moreover, condition \eqref{assumption2} is satisfied also by certain convex but not horoconvex domains, see Section \ref{examples}.

The geometric conditions \eqref{assumption} and \eqref{assumption2} have the common form
\begin{equation}\label{static}
V k_{\Omega}-\partial_{\nu}V>0
\end{equation}
with
$$
V=
\begin{cases}
\cos(\theta)&\hbox{ in }\mathbb S^2\\
1&\hbox{ in }\mathbb R^2\\
\cosh(r)&\hbox{ in }\mathbb H^2
\end{cases}
$$
Now, let  $(M,g)$ be an $n$-dimensional Riemannian manifold endowed with a non-trivial smooth potential $V$. Then $(M,g,V)$ is called a {\it static triple} if 
$$\Delta V\,g-D^2V+V\,{\rm Ric}\equiv 0\hbox{ in }M$$ 
The space forms $\mathbb S^2,\mathbb R^2,\mathbb H^2$ with   potentials $V=\cos(\theta),1,\cosh(r)$, respectively, are static triples. Note that in these cases $V$ depends on the distance from a fixed base point $P$. The potential $V$ is called a {\it static potential} for the manifold. The study of Riemannian manifolds admitting static potentials has deep roots in Riemannian geometry and in particular, in general relativity. In this context, it has been pointed out that one natural notion of convexity for a domain $\Omega$ is the following:
$k_i\geq\frac{\partial_{\nu}V}{V}$ on $\partial\Omega$, where $k_i$, $i=1,...,n-1$, are the principal curvatures (provided $V>0$ on $\partial\Omega$). This condition appears, under the name {\it static convexity}, in the work of Brendle and Wang \cite{brendle} on the Gibbons-Penrose inequality in Schwarzschild spacetime.  Static convexity has since appeared naturally in many other related contexts,  for instance in the  Minkowski and Alexandrov-Fenchel type inequalities \cite{lixia}, and in the study of geometric flows \cite{huli}, where static convexity is shown to be preserved along suitable curvature flows in hyperbolic space.

Condition \eqref{static} can be also interpreted as follows. Suppose that $V>0$ in $\Omega$. Then $Vk_{\Omega}-\partial_{\nu}V$ is the geodesic curvature of $\partial\Omega$ for the conformal metric $V^{-2}g$, where $g$ in our case is the spherical, euclidean or hyperbolic metric.
 
For a domain $\Omega\subset\mathbb H^2$, {\it condition \eqref{assumption2} is telling that the geodesic curvature of $\partial\Omega$ is positive for the hyperbolic metric rescaled by $\cosh(r)^{-2}$}. This metric is exactly  the spherical metric on a hemisphere.

 On the other hand, if $\Omega\subset \mathbb S^2$ is contained in a hemisphere, condition \eqref{assumption} amounts to  the positivity of the geodesic curvature of $\partial\Omega$ for the hyperbolic metric. If $\Omega$ is not contained in the hemisphere
$\mathbb S^2_+=\{\theta<\pi/2\}$, the discussion above applies to $\Omega\cap\mathbb S^2_+$ and $\mathbb S^2_-\setminus\overline\Omega$, where $\mathbb S^2_-=\{\theta>\pi/2\}$: both are domains, possibly having part of
their boundary at infinity, in the corresponding hyperbolic planes.

Equivalently, when $V>0$, condition \eqref{static} can be read as the positivity of a weighted curvature. More precisely, on an $n$-dimensional weighted Riemannian manifold $(M,g,V\,dv_g)$, if $\Sigma$ is an hypersurface with mean curvature $H$, the quantity $H-\frac{\partial_\nu V}{V}$ is referred to as the weighted mean curvature, see e.g., \cite{IRSavo}. In our case, the weight is the static potential of the manifold.

Condition \eqref{static} suggests that the results of Theorems \ref{main} and \ref{main2} may be extended to static convex domains in static triples. However, such an extension turns out to be far from straightforward.

The second part of the statements of Theorems \ref{main} and \ref{main2} shows that the geometric conditions \eqref{assumption} and \eqref{assumption2}, which ensure the uniqueness of the critical point of $u$, are indeed optimal. We focus on the case of the torsion problem (i.e., $f \equiv 1$ in \eqref{eq1}) and construct examples of solutions to \eqref{eq1} in a family of geodesically star-shaped domains of $M = \mathbb{S}^2$ or $\mathbb{H}^2$, which exhibit an arbitrary number of critical points and satisfy \eqref{assumptionmf} and \eqref{assumptionmf2}.

This type of examples first appeared in \cite{GG22} and \cite{DRG1}, where the authors demonstrated the optimality of the convexity assumption on $\Omega$ in the result of \cite{CC} when $\Omega \subset \mathbb{R}^2$. The example in \cite{GG22} was later extended in \cite{EGG25} to the setting of a general $d$-dimensional Riemannian manifold $M$, as well as to equations involving general nonlinearities $f(u)$.

Here, we construct more explicit examples of solutions to \eqref{eq1}, which preserve all the key features of those in \cite{GG22} and \cite{EGG25}, directly within $\mathbb{S}^2$ and $\mathbb{H}^2$. These examples are, in a sense, pathological. We shall see that the associated domains converge to suitable geodesic segments. Such pathological behavior of the domain is indeed necessary, since small perturbations of a convex domain do not produce additional critical points, as shown in \cite{BDM}.

We remark that the fact that the limiting configuration is built around a geodesic  appears to be a necessary condition for this type of construction, and we believe that it holds in a more general setting. We intend to investigate this phenomenon in future papers.

\subsection{Strategy of the proof} We give here some ideas about the proof of Theorem \ref{main}, namely, of the {\it spherical case}. The proof of Theorem \ref{main2} follows the same lines (it is actually simpler). 

The proof of Theorem \ref{main} is reminiscent of \cite{GP1}, however it differs in many aspects. As in \cite{GP1}, the most important tool that we use is the Poincaré-Hopf Theorem, which relates the index of the zeros of any vector field $V$ on a differentiable manifold with its topology through the Euler characteristic of the manifold, see Theorem \ref{PH}. In our setting, we are dealing with contractible sets, hence Theorem \ref{PH} gives
\begin{equation}\label{PH0}
\sum_i{\rm Ind}_{p_i}V=1,
\end{equation}
where the sum runs on the {\it isolated zeros} of $V$. One would be tempted to use \eqref{PH0} with $V=\nabla u$ and recover somehow a counting of the critical points of $u$. Unfortunately, it could be difficult to work directly on $\nabla u$ and check if its zeros are singular and what is their index. Following some ideas of \cite{GP1}, we consider here an auxiliary vector field $V$ for which we have a way to determine the non-degeneracy and the index of the corresponding zeros, and hence we can apply \eqref{PH0}. We then relate the zeros of $V$ to the zeros of $\nabla u$ and conclude.

The vector field $V$ is obtained by combining in a suitable way $K_1,K_2$ and $\nabla u$, where $K_1,K_2$ are two Killing fields on the sphere (see \eqref{V0} and \eqref{V1}). Extrinsically, $K_1,K_2$ can be seen as the restriction to $\mathbb S^2$ of the fields $z\partial_x-x\partial_z$ and $z\partial_y-y\partial_z$ of $\mathbb R^3$, respectively.

The proof of Theorem \ref{main} goes roughly as follows:

\begin{itemize}
\item[$i)$] We first observe that critical points of $u$ are non-degenerate. Assuming that a critical point $p$ of $u$ is degenerate, we can define an auxiliary function $Z:\Omega\to\mathbb R$ which vanishes along with its gradient at $p$. Then $p$ is a singular point of $Z$, and this implies that the zero level set of $Z$ is given by the union of at least two curves intersecting transversally at $p$. We also show that condition \eqref{assumption} in Theorem \ref{main} implies that $Z$ must have exactly two zeros at $\partial\Omega$. We reach a contradiction since now $Z$ turns out to be a Dirichlet eigenfunction of $-\Delta-f'(u)$ with eigenvalue $0$ on a proper subdomain of $\Omega$, violating the non-negativity of  $-\Delta-f'(u)$ in $\Omega$. With the same argument we prove that there are no critical points on the equator of the sphere with pole $P$ (if the intersection of $\Omega$ with the equator is nonempty).
\item[$ii)$] The previous step allows to prove that the critical points of $u$ coincide with the zeros of $V$ not belonging to the equator. One implication is straightforward from the definition of $V$ (see \eqref{V0}, \eqref{V1}). The proof of the second implication is in the same spirit of point $i)$, using a suitable (possibly different) auxiliary function $Z$. Here we note a crucial difference with respect to \cite{GP1}: the vector field $V$ may have zeros that are not critical points of $u$ and that must be carefully considered; this does not happen in the Euclidean case and in the hyperbolic case. This fact will have also important consequences on the location of critical points for convex spherical sets, which is a new information provided by the spherical geometry, see Subsection \ref{sub:location}.
\item[$iii)$] We relate the index of $p$ as zero of $\nabla u$ to that of $p$ as zero of $V$ (recall that $V$ and $\nabla u$ share the same zeros outside the equator from $ii)$), and prove that ${\rm Ind}_pV=1$ for {\bf all}  such zeros. We also prove that ${\rm Ind}_pV=1$ also for the zeros of $V$ on the equator. 
\item[$iv)$] Finally we prove that condition \eqref{assumption} in Theorem \ref{main} implies $\langle V,\nu\rangle>0$ on $\partial\Omega$. We are then in position to apply the Poincaré-Hopf Theorem to $V$: from \eqref{PH0} we get
$$
1\leq \#\{{\rm critical\ points\ of\ }u\}\leq\sum_i{\rm Ind}_{p_i}V=\chi(\Omega)=1,
$$
which gives the uniqueness of the zero of $V$ and its non-singularity. This is equivalent to the uniqueness of the critical point of $u$ and its non-degeneracy.
\end{itemize}
Some of the previous points make use of ideas introduced by Cabrè and Chanillo in \cite{CC}.

The present paper is organized as follows. In Section \ref{pre} we collect a few preliminary results on vector fields on the sphere, on the Poincaré-Hopf Theorem and on the behavior of solutions of elliptic PDE's near singular zeros. The proof of the first part of Theorem \ref{main} (concerning the uniqueness of the critical point under condition \eqref{assumption}) is presented in Section \ref{proof}, while the proof of the first part of Theorem \ref{main2} (concerning the uniqueness of the critical point under condition \eqref{assumption2}) is presented in section \ref{proof2}.  In Section \ref{s10} we prove the optimality of  conditions \eqref{assumption} and \eqref{assumption2},  completing the proofs of Theorems \ref{main} and \ref{main2}. Finally, in Section \ref{examples} we discuss some classes of domains for which Theorems \ref{main} and \ref{main2} apply. In particular, in Subsection \ref{sub:location} we provide information on the location of the critical point for convex spherical domains of diameter larger than $\pi/2$.

\section{Preliminaries and notation}\label{pre}
In this section we collect a few preliminary results and we fix the notation. 

\subsection{Vector fields}

In this paper we consider as ambient manifolds the sphere $\mathbb S^2$ and the hyperbolic plane $\mathbb H^2$  endowed with the standard metrics of constant curvature $1$ and $-1$, respectively. In particular, $\mathbb S^2=\{(x,y,z)\in\mathbb R^3:x^2+y^2+z^2=1\}$ and the round metric on $\mathbb S^2$ is the one induced by the euclidean metric on $\mathbb R^3$. Sometimes, it will be convenient to look extrinsically at $\mathbb S^2$ as a submanifold of $\mathbb R^3$. We shall denote by $\langle\cdot,\cdot\rangle$ the inner product on tangent spaces on $\mathbb S^2$ and $\mathbb H^2$ associated with the chosen standard metrics.

\subsubsection{Killing vector fields}\label{fields}
We recall that a {\it Killing} vector field $K$ on a Riemannian manifold $M$ is characterized by the fact that the Lie derivative of the metric with respect to $K$ vanishes. This is equivalent to
\begin{equation}\label{antisym}
\langle\nabla_XK,Y\rangle+\langle\nabla_YK,X\rangle=0
\end{equation}
for all vector fields $X,Y$. If $K$ is a Killing field, then ${\rm div}\,K=0$; if a Killing field $K$ is pointwise tangent to an embedded submanifold, then its restriction is a Killing field for the submanifold and for the induced metric; $K$ is a Killing vector field if and only if for any smooth function $u:M\to\mathbb R$ 
\begin{equation}\label{commuting}
\Delta(\langle K,\nabla u\rangle)=\langle K,\nabla\Delta u\rangle,
\end{equation} 
that is, $K$ commutes with the Laplacian.
We refer to \cite{petersen} for more information on Killing vector fields.

We consider on $\mathbb S^2$ the two Killing vector fields $K_1$ and $K_2$ given respectively by the restriction of the Killing fields $z\partial_x-x\partial_z$ and  $z\partial_y-y\partial_z$ of $\mathbb R^3$ to $\mathbb S^2$. Note that each $K_i$ has exactly two zeros, which are antipodal points on the equator of $\mathbb S^2$. Using spherical coordinates $(\theta,\phi)\in[0,\pi]\times[0,2\pi]$ on $\mathbb S^2$, where $\theta$ is the geodesic distance from the north pole $(0,0,1)$, we can write $K_1,K_2$ as
\begin{itemize}
\item[\textbullet] $K_1=\cos(\phi)\partial_{\theta}-\cot(\theta)\sin(\phi)\partial_{\phi}$; the integral curves are spherical circles centered at $(0,\pm 1,0)$;
\item[\textbullet] $K_2=\sin(\phi)\partial_{\theta}+\cot(\theta)\cos(\phi)\partial_{\phi}$; the integral curves are spherical circles centered at $(\pm 1,0,0)$.

\end{itemize}

We consider the Poincaré disk model for $\mathbb H^2$. Namely, we consider $D$ to be the open unit disk of $\mathbb R^2$ with Cartesian coordinates $(x,y)$ endowed with the conformal metric $\frac{4}{(1-x^2-y^2)^2}(dx^2+dy^2)$. On $\mathbb H^2$ we consider the two Killing vector fields $K_1$ and $K_2$ given by 
\begin{itemize}
\item[\textbullet] $K_1=\frac{1-x^2+y^2}{2}\partial_x-xy\partial_y$; the integral curves are the intersection of $D$ with arcs of circles  with centers on $x=0$ and passing through $(\pm 1,0)$;
\item[\textbullet] $K_2=-xy\partial_x+\frac{1+x^2-y^2}{2}\partial_y$; the integral curves are the intersection of $D$ with arcs of circles  with centers on $y=0$ and passing through $(0,\pm 1)$.

\end{itemize}

\subsubsection{Orientations and Hodge star operator}
We choose on $\mathbb S^2\subset\mathbb R^3$ the orientation given by the normal vector field $N$ which is outward pointing to $\mathbb S^2$, namely, $N=(x,y,z)$. Alternatively, this is the orientation inherited from the positive orientation of $\mathbb R^2$ through the stereographic projection from the south pole $(0,0,-1)$.  Let $p\in\mathbb S^2$ and take $(\vec{e_1},\vec{e_2})$ a positively oriented orthonormal basis of $T_p\mathbb S^2$. Let $(dx_1,dx_2)$ its dual basis in $T^*_p\mathbb S^2$. Then the Hodge star operator $\star$ acts as follows: $\star dx_1=dx_2$, $\star dx_2=-dx_1$. By abuse of notation, we shall identify a vector field $V$ with its dual one-form. Then we define the vector field $V^{\perp}$ as
\begin{equation}\label{perp}
V^{\perp}:=\star V.
\end{equation}
Extrinsically, $V^{\perp}$ is obtained by pointwise rotating $V$ counter-clockwise of an angle of $\pi/2$ around $N$. It can also be defined (extrinsically) as
$$
V^{\perp}=N\times V,
$$
where $\times$ is the standard cross product of $\mathbb R^3$.

Analogously, we choose on $\mathbb H^2$ the orientation inherited by the positive orientation of $D\subset\mathbb R^2$. Namely, $(\partial_x,\partial_y)$ is positively oriented. For a vector field $V$ on $\mathbb H^2$, the field $V^{\perp}$ is defined as in the case of the sphere (and can be thought as a pointwise counter-clockwise rotation of $V$ of $\frac{\pi}{2}$ in the disk model).

\subsection{Poincaré-Hopf Theorem} We state here the Poincaré-Hopf Theorem which relates the index of a vector field with the Euler characteristic of the underlying manifold. We recall that for a vector field $V$ on a differentiable manifold, the index ${\rm Ind}_{p_i}V$ at an isolated zero $p_i$ of $V$ is defined to be the degree of the self-map of $\mathbb S^{n-1}$ obtained by normalizing the vector field image of $V$  by a chart, restricted to a small sphere around the image of $p_i$.  We have the following

\begin{thm}[Poincaré-Hopf]\label{PH}
Let $M$ be a $n$-dimensional Riemannian manifold and let $V$ be a vector field on $M$ with isolated zeros $p_i$. If $\partial M\ne\emptyset$, assume that $\langle V,\nu\rangle>0$, where $\nu$ is the unit outer normal to $\partial M$. Then
$$
\sum_i{\rm Ind}_{p_i}V=\chi(M),
$$
where $\chi(M)$ is the Euler characteristic of $M$.
\end{thm}
We refer e.g., to \cite{milnor} for precise definitions and for a proof of Theorem \ref{PH}. We conclude this subsection by noting that if a vector field $V$ on a manifold $M$ has a zero $p$ which is non-degenerate, i.e., ${\rm det\,}(\partial_{x_j}V_i(0))\ne 0$, then ${\rm Ind}_pV={\rm sign}\,{\rm det}(\partial_{x_j}V_i(0))$. Here $(x_1,...,x_n)$ is a local system of coordinates around $p$ (which corresponds to $0\in\mathbb R^n$) and $V=\sum_{i=1}^nV_i\partial_{x_i}$.

\subsection{Nodal sets of solutions of elliptic PDE's}

We recall here the following classical result on the behavior of solutions to elliptic equations on domains of $\mathbb S^2$ and $\mathbb H^2$ at a singular zero, namely, at a point $q$ where $v(q)=0$ and $\nabla v(q)=0$. For a proof, we refer to \cite{HS}.
\begin{thm}\label{nod1}
Let $v$ be a solution of $-\Delta v=a v$ on $\Omega\subset \mathbb S^2,\mathbb H^2$, where $a$ is a smooth function. Let $q\in\Omega$ be such that $v(q)=\nabla v(q)=0$. Then, in a neighborhood of $q$ the set $v^{-1}(0)$ is given by (at least) two curves which intersect transversely.
\end{thm}

\section{Proof of Theorem \ref{main}  (uniqueness of the critical point)}\label{proof}

In this section we will prove Theorem \ref{main} in several steps. We start by noting that by hypothesis $f(0)\geq 0$ and this implies, by Hopf's lemma, that $|\nabla u|\ne 0$ on $\partial\Omega$. In particular, since $u>0$ in $\Omega$, we have that the unit outer normal to $\partial\Omega$ is given by
\begin{equation}\label{3.1}
\nu=-\frac{\nabla u}{|\nabla u|}.
\end{equation}

By assumption $\Omega$ is uniformly star-shaped with respect to $P\in\Omega$. We assume from now on that $P$ is the north pole for a system of spherical coordinates $(\theta,\phi)\in[0,\pi]\times[0,2\pi]$, where $\theta$ is the geodesic distance from $P$.

\subsection{The auxiliary vector field $V$}

We introduce now the following vector field:
\begin{equation}\label{V0}
V=-\tilde V^{\perp},
\end{equation}
where
\begin{equation}\label{V1}
\tilde V=\langle K_1,\nabla u\rangle\nabla(\langle K_2,\nabla u\rangle)-\langle K_2,\nabla u\rangle\nabla(\langle K_1,\nabla u\rangle).
\end{equation}
Recall that $K_1,K_2$ are the two spherical Killing fields defined in Subsection \ref{fields} and $^{\perp}$ is defined in \eqref{perp}.

\subsection{Critical points of $u$ are non degenerate. There are no critical points on the equator.}\label{nondegenerate} 

Assume first that the critical point $p$ does not belong to the equator, that is,  $\theta\ne\frac{\pi}{2}$. Suppose that $p$ is such that $\nabla u(p)=0$ and $D^2u_{|_p}(X,\cdot)=0$ for some $X\in T_p\mathbb S^2$. Without loss of generality we may assume $X=K_1(p)$. In fact $K_1,K_2$ are linearly independent and do not vanish at any $p$ not on the equator, so we can always find a linear combination of the two that is parallel to $X$ at $p$. Moreover any combination $aK_1+bK_2$ with $a^2+b^2=1$ is just a rotation of $K_1$ (or $K_2$) around the north pole $P$.  We define then $Z=\langle K_1,\nabla u\rangle$. We have that 
 $$
 -\Delta Z=f'(u)Z
 $$ in $\Omega$ by \eqref{commuting}, since $K_1$ is a Killing field, and by construction, $Z(p)=0$ and $\nabla Z(p)=0$. 
 
Assume now that $p$ is on the equator and that $\nabla u(p)=0$. Then, without loss of generality, we can also assume that $K_1(p)=0$ (again, this can be obtained with a rotation of the coordinate system around the north pole $P$). We define $Z=\langle K_1,\nabla u\rangle$ as above, and immediately check that $-\Delta Z=f'(u)Z$, $Z(p)=0$ and $\nabla Z(p)=0$. Note that in this case we haven't supposed that the critical point is degenerate.
 
We will prove in Subsection \ref{twozeros} that under assumption \eqref{assumption} $Z$ has exactly two zeros on $\partial\Omega$ and, in particular, $Z$ does not vanish identically in $\Omega$. Therefore, in a neighborhood of $p$, the set $Z=0$ is given by (at least) two curves which intersect transversally. Since $Z$ has exactly two zeros on $\partial\Omega$, necessarily the set $Z=0$ creates a loop, i.e., there exists a proper nodal domain $\omega\subsetneq\Omega$, with $\overline\omega$ possibly meeting $\partial\Omega$ at the two boundary zeros of $Z$. In particular, $Z=0$ on $\partial\omega$ and $Z$ does not change sign in $\omega$, and solves $-\Delta Z=f'(u)Z$ in $\omega$. By domain monotonicity, the first eigenvalue of $-\Delta-f'(u)$ in $\Omega$ is negative, and this is a contradiction with the semi-stability of $u$ \eqref{semistable}.

We conclude that critical points of $u$ are non-degenerate, and in particular, they are isolated (and in finite number). Moreover, there are no critical points on the equator $\theta=\frac{\pi}{2}$.

\medskip

\subsection{The zeros of $V$ not on the equator coincide with the zeros of $\nabla u$}\label{zerosequal} In this subsection we prove that $\nabla u$ and $V$ share the same zeros in $\Omega\setminus\{\theta=\frac{\pi}{2}\}$.
Actually, we prove the equivalent statement that zeros of $\nabla u$ coincide with the zeros of $\tilde V$ not belonging to the equator. If $\nabla u(p)=0$, then by construction $\tilde V(p)=0$. Assume now that $p$ does not belong to the equator and that $\tilde V(p)=0$ but $\nabla u(p)\ne 0$.


Let
\begin{equation}\label{maxK}
K:=a K_2+b K_1
\end{equation}
where
$$
a=\langle K_1(p),\nabla u(p)\rangle\,,\ \ \ b=-\langle K_2(p),\nabla u(p)\rangle.
$$
Note that, since $\nabla u(p)\ne 0$ and $K_1,K_2$ are linearly independent and do not vanish at all $p$ not on the equator, $(a,b)\ne (0,0)$, hence $K$ is not identically zero.

Define $Z=\langle K,\nabla u\rangle$. We have that 
$$
-\Delta Z=f'(u)Z
$$
since $K$ is a Killing field, and that $Z(p)=0$ by construction. Moreover
$$
\nabla Z(p)=\tilde V(p)=0.
$$
We are led to a contradiction as in Subsection \ref{nondegenerate}.

\subsection{The vector field $V$ satisfies $\langle V,\nu\rangle>0$ on $\partial\Omega$}\label{boundaryC}
We start by noting that
$$
\langle V,\nu\rangle_{|_{\partial\Omega}}=-\frac{1}{|\nabla u|}\langle V,\nabla u\rangle=-\frac{1}{|\nabla u|}\langle V^{\perp},\nabla^{\perp} u\rangle=-\frac{1}{|\nabla u|}\langle\tilde V,\nabla^{\perp} u\rangle,
$$
where we have used the fact that $(\tilde V^{\perp})^{\perp}=-\tilde V$. Here by $\nabla^{\perp}u$ we denote the vector field $(\nabla u)^{\perp}$.

Now we can write
\begin{align}
-\langle\tilde V,\nabla^{\perp}u\rangle
=&\langle K_1,\nabla u\rangle DK_2(\nabla u,\nabla^{\perp}u)-\langle K_1,\nabla u\rangle D^2u(K_2,\nabla^{\perp}u)\nonumber\\
-&\langle K_2,\nabla u\rangle DK_1(\nabla u,\nabla^{\perp}u)+\langle K_2,\nabla u\rangle D^2u(K_1,\nabla^{\perp}u)\nonumber\\
=&\langle K_1,\nabla u\rangle DK_2(\nabla u,\nabla^{\perp}u)-\langle K_2,\nabla u\rangle DK_1(\nabla u,\nabla^{\perp}u)\label{red}\\
-&\langle K_1,\nabla u\rangle D^2u(K_2,\nabla^{\perp}u)+\langle K_2,\nabla u\rangle D^2u(K_1,\nabla^{\perp}u)\label{blue}.
\end{align}
where in the second inequality we have just re-arranged the terms. 

We consider \eqref{blue} first. Writing $K_i=\langle K_i,\frac{\nabla u}{|\nabla u|}\rangle\frac{\nabla u}{|\nabla u|}+\langle K_i,\frac{\nabla^{\perp}u}{|\nabla u|}\rangle\frac{\nabla^{\perp}u}{|\nabla u|}$, plugging this \eqref{blue}, and collecting $|\nabla u|^2$, \eqref{blue} simplifies to
\begin{equation}\label{blue12}
-|\nabla u|^2\left(\langle K_1,\frac{\nabla u}{|\nabla u|}\rangle\langle K_2,\frac{\nabla^{\perp}u}{|\nabla u|}\rangle-\langle K_2,\frac{\nabla u}{|\nabla u|}\rangle\langle K_1,\frac{\nabla^{\perp}u}{|\nabla u|}\rangle\right)D^2u\left(\frac{\nabla^{\perp}u}{|\nabla u|},\frac{\nabla^{\perp}u}{|\nabla u|}\right).
\end{equation}
We have that 
$$
D^2u\left(\frac{\nabla^{\perp}u}{|\nabla u|},\frac{\nabla^{\perp}u}{|\nabla u|}\right)=\Delta u-\partial^2_{\nu\nu}u=\kappa\partial_{\nu}u=-\kappa|\nabla u|
$$
while
$$
\langle K_1,\frac{\nabla u}{|\nabla u|}\rangle\langle K_2,\frac{\nabla^{\perp}u}{|\nabla u|}\rangle-\langle K_2,\frac{\nabla u}{|\nabla u|}\rangle\langle K_1,\frac{\nabla^{\perp}u}{|\nabla u|}\rangle
$$
is the {\it signed} area of the parallelogram formed by $K_1,K_2$ (in this order, since $(\nabla u,\nabla^{\perp}u)$ is positively oriented). With our choice of $K_1,K_2$ we see that this sign is {\it positive} on the upper hemisphere and {\it negative} in the lower hemisphere, while this area vanishes at the equator. We have proved that \eqref{blue} can be written as
$$
|\nabla u|^3A(K_1,K_2)\kappa,
$$
 An explicit computation allows to establish that $A(K_1,K_2)=\cos(\theta)$.

Now we proceed to study \eqref{red}. To do so, we write (if $K_1\ne 0$)
\begin{multline*}
DK_2(\nabla u,\nabla^{\perp}u)\\
=DK_2\left(\langle\nabla u,\frac{K_1}{|K_1|}\rangle\frac{K_1}{|K_1|}+\langle\nabla u,\frac{K_1^{\perp}}{|K_1|}\rangle\frac{K_1^{\perp}}{|K_1|},\langle\nabla^{\perp}u,\frac{K_1}{|K_1|}\rangle\frac{K_1}{|K_1|}+\langle\nabla^{\perp}u,\frac{K_1^{\perp}}{|K_1|}\rangle\frac{K_1^{\perp}}{|K_1|}\right)\\
=\langle\nabla u,\frac{K_1}{|K_1|}\rangle\langle\nabla^{\perp}u,\frac{K_1^{\perp}}{|K_1|}\rangle DK_2\left(\frac{K_1}{|K_1|},\frac{K_1^{\perp}}{|K_1|}\right)\\
+\langle\nabla u,\frac{K_1^{\perp}}{|K_1|}\rangle\langle\nabla^{\perp}u,\frac{K_1}{|K_1|}\rangle DK_2\left(\frac{K_1^{\perp}}{|K_1|},\frac{K_1}{|K_1|}\right)\\
=\left(\left|\langle\nabla u,\frac{K_1}{|K_1|}\rangle\right|^2+\left|\langle\nabla u,\frac{K_1^{\perp}}{|K_1|}\rangle\right|^2\right)DK_2\left(\frac{K_1}{|K_1|},\frac{K_1^{\perp}}{|K_1|}\right)\\
=|\nabla u|^2DK_2\left(\frac{K_1}{|K_1|},\frac{K_1^{\perp}}{|K_1|}\right).
\end{multline*}
Here we have used the fact that $DK_2$ is anti-symmetric and that $\langle U^{\perp},V\rangle=-\langle U,V^{\perp}\rangle$ and $U^{\perp\perp}=-U$ for tangent vectors $U,V$. Analogously we get (if $K_2\ne 0$)
$$
DK_1(\nabla u,\nabla^{\perp}u)=|\nabla u|^2 DK_1\left(\frac{K_2}{|K_2|},\frac{K_2^{\perp}}{|K_2|}\right)
$$
Hence \eqref{red} is rewritten as
$$
|\nabla u|^2\langle F,\nabla u\rangle
$$
where
$$
F=DK_2\left(\frac{K_1}{|K_1|},\frac{K_1^{\perp}}{|K_1|}\right)K_1-DK_1\left(\frac{K_2}{|K_2|},\frac{K_2^{\perp}}{|K_2|}\right)K_2,
$$
which is a combination of the two Killing fields $K_1,K_2$. Note that this is also true at the points where $K_1$ and $K_2$ vanish (they do not vanish at the same points).

Now, we need to use the explicit expressions of $K_1,K_2$. Recall that $K_1=z\partial_x-x\partial_z$ and $K_2=z\partial_y-y\partial_z$. We compute that $|K_1|^2=x^2+z^2$ and $|K_2|^2=y^2+z^2$ and that
$$
DK_2(K_1,\cdot)=-x\partial_y\,,\ \ \ DK_1(K_2,\cdot)=-y\partial_x.
$$
We also check that $K_1^{\perp}=-xy\partial_x+(x^2+z^2)\partial_y-yz\partial_z$ so that
$$
DK_2\left(\frac{K_1}{|K_1|},\frac{K_1^{\perp}}{|K_1|}\right)=-x.
$$
Analogously, since $K_2^{\perp}=(-y^2-z^2)\partial_x+xy\partial_y+xz\partial_z$, we see that
$$
DK_1\left(\frac{K_2}{|K_2|},\frac{K_2^{\perp}}{|K_2|}\right)=y.
$$
Therefore
$$
F=-xK_1-yK_2=-xz\partial_x-yz\partial_y+(x^2+y^2)\partial_z.
$$
However this is not the most convenient way to look at $F$. In spherical coordinates $(\theta,\phi)\in[0,\pi]\times[0,2\pi]$ we have $K_1=\cos(\phi)\partial_{\theta}-\cot(\theta)\sin(\phi)\partial_{\phi}$ and $K_2=\sin(\phi)\partial_{\theta}+\cot(\theta)\cos(\phi)\partial_{\phi}$. Since $x=\sin(\theta)\cos(\phi)$ and $y=\sin(\theta)\sin(\phi)$ we have
$$
F=-\sin(\theta)\partial_{\theta}.
$$
We conclude that \eqref{red} can be re-written as
$$
-|\nabla u|^2\sin(\theta)\langle\nabla u,\vec{e_ {\theta}}\rangle=|\nabla u|^3\sin(\theta)\langle\nu,\vec{e_ {\theta}}\rangle.
$$

We have proved that
\begin{equation}\label{expression}
\langle V,\nu\rangle=|\nabla u|^2\left( \cos(\theta)\kappa+\sin(\theta)\langle\nu,\vec{e_{\theta}}\rangle\right)
\end{equation}
on $\partial\Omega$, which is strictly positive by the assumption \eqref{assumption}.

\subsection{The index of any zero of $V$ is $1$}\label{indexS} 

There are two families of zeros of $V$: the zeros of $V$ in $\Omega\setminus\{\theta=\frac{\pi}{2}\}$, which coincide with the zeros of $\nabla u$, by Subsections \ref{nondegenerate} and \ref{zerosequal}; the zeros of $V$ in $\Omega\cap \{\theta=\frac{\pi}{2}\}$. 

Now, we observe that the zeros $p$ of $V$ in $\Omega\cap \{\theta=\frac{\pi}{2}\}$ are characterized by $\langle\nabla u(p),\vec{e_{\theta}}\rangle=0$, which implies that $\langle K_i(p),\nabla u(p)\rangle=0$. In fact, if $p$ on the equator is such that $\langle\nabla u(p),\vec{e_{\theta}}\rangle\ne 0$,  as in Subsection \ref{zerosequal} we can find a linear combination $K$ of $K_1,K_2$ such that $\langle K(p),\nabla u(p)\rangle=0$ (one checks that $K(p)=0$) and define $Z=\langle K,\nabla u\rangle$. It follows that $Z(p)=0$, and if we assume that $V(p)=0$, we have also $\nabla Z(p)=0$; since $-\Delta Z-f'(u)Z=0$ in $\Omega$, we reach a contradiction as in Subsection \ref{zerosequal}.

We have to compute the index of the zeros in any of such families.

Let $p$ be a zero of $\nabla u$, which means that it is a zero of $V$ in $\Omega\setminus\{\theta=\frac{\pi}{2}\}$. Let $(\vec{e_1},\vec{e_2})$ be an orthonormal basis of $T_p\mathbb S^2$. We have that
\begin{multline*}
\nabla_{\vec{e_1}}V(\vec{e_1})\\
=-\langle\nabla(\langle\nabla u,K_1\rangle,\vec{e_1}\rangle)\langle\nabla^{\perp}(\langle\nabla u,K_2\rangle),\vec{e_1}\rangle+\langle\nabla(\langle\nabla u,K_2\rangle,\vec{e_1}\rangle)\langle\nabla^{\perp}(\langle\nabla u,K_1\rangle),\vec{e_1}\rangle\\
=\langle\nabla(\langle\nabla u,K_1\rangle,\vec{e_1}\rangle)\langle\nabla(\langle\nabla u,K_2\rangle),\vec{e_2}\rangle-\langle\nabla(\langle\nabla u,K_2\rangle,\vec{e_1}\rangle)\langle\nabla(\langle\nabla u,K_1\rangle),\vec{e_2}\rangle\\
=-D^2u(K_2,\vec{e_1})D^2u(K_1,\vec{e_2})+D^2u(K_1,\vec{e_1})D^2u(K_2,\vec{e_2}).
\end{multline*}
In the same way we see that
$$
\nabla_{\vec{e_2}}V(\vec{e_2})=-D^2u(K_2,\vec{e_1})D^2u(K_1,\vec{e_2})+D^2u(K_1,\vec{e_1})D^2u(K_2,\vec{e_2})
$$
and 
$$
\nabla_{\vec{e_2}}V(\vec{e_1})=\nabla_{\vec{e_1}}V(\vec{e_2})=0.
$$
Moreover, writing $K_i=\langle K_i,\vec{e_1}\rangle \vec{e_1}+\langle K_i,\vec{e_2}\rangle \vec{e_2}$ in the expression for $\nabla_{\vec{e_1}}V(\vec{e_1})$ (and $\nabla_{\vec{e_2}}V(\vec{e_2})$ which is the same), we get
\begin{multline*}
\nabla_{\vec{e_1}}V(\vec{e_1})=\nabla_{\vec{e_2}}V(\vec{e_2})\\=\left(\langle K_1,\vec{e_1}\rangle\langle K_2,\vec{e_2}\rangle-\langle K_2,\vec{e_1}\rangle\langle K_1,\vec{e_2}\rangle\right)\left(D^2u(\vec{e_1},\vec{e_1})D^2u(\vec{e_2},\vec{e_2})-D^2u(\vec{e_1},\vec{e_2})^2\right)\\
=A(K_1,K_2){\rm det}D^2u=\cos(\theta){\rm det}D^2u.
\end{multline*}
Since $p$ does not belong to the equator $\theta=\frac{\pi}{2}$, $\cos(\theta)\ne 0$. Hence $p$ is a non-degenerate zero of $V$ of index $1$.

We investigate now the second family of zeros of $V$. Let $p\in\Omega\cap\{\theta=\frac{\pi}{2}\}$ be such that $\langle \nabla u(p),\vec{e_{\theta}}\rangle=0$, which implies $V(p)=0$. As remarked at the beginning of this section, these are the only zeros of $V$ on the equator. We consider at $p$ the following orthonormal basis of $T_p\mathbb S^2$: $(\vec{e_{\theta}},\vec{e_{\phi}})$. Extrinsically, $\vec{e_{\theta}}=(0,0,-1)$, $\vec{e_{\phi}}=(-y_p,x_p,0)$, where $p=(x_p,y_p,0)$. This basis is positively oriented by our choice of orientation of $\mathbb S^2$.

We repeat the same computation of the previous paragraph, using the fact that $\nabla u=\pm|\nabla u|\vec{e_{\phi}}$ at $p$, and the anti-symmetry of $DK_i$. We find that
\begin{itemize}
\item $\nabla_{\vec{e_{\theta}}}V(\vec{e_{\phi}})=\nabla_{\vec{e_{\phi}}}V(\vec{e_{\theta}})=0$;
\item $\nabla_{\vec{e_{\theta}}}V(\vec{e_{\theta}})=\nabla_{\vec{e_{\phi}}}V(\vec{e_{\phi}})=\pm|\nabla u(p)|D^2u(p)(\vec{e_{\theta}},\vec{e_{\phi}})$.
\end{itemize}
Now, since $|\nabla u(p)|\ne 0$, we have that if $D^2u(p)(\vec{e_{\theta}},\vec{e_{\phi}})\ne 0$ the index of $V$ at $p$ is $1$.

We show now that necessarily $D^2u(p)(\vec{e_{\theta}},\vec{e_{\phi}})\ne 0$, and this will be done by contradiction, using the same argument of Subsections \ref{nondegenerate} and \ref{zerosequal}. Assume that $D^2u(p)(\vec{e_{\theta}},\vec{e_{\phi}})= 0$. This means that $(\vec{e_{\theta}},\vec{e_{\phi}})$ is a orthonormal basis of $T_p\mathbb S^2$ of eigenvectors of $D^2u(p)$. In particular there exists $\lambda\in\mathbb R$ such that $D^2u(p)(\vec{e_{\theta}})=\lambda\vec{e_{\theta}}$. We have two possibilities: $\lambda=0$ and $\lambda\ne 0$.
\begin{itemize}
\item Assume that $\lambda=0$. Since $V$ is no longer involved, it is easier to rotate the system of coordinates and assume that $p=(0,1,0)$. Take then $Z=\langle K_2,\nabla u\rangle$. In particular, at $p$, $K_2(p)=\vec{e_{\theta}}$. By construction, $Z(p)=0$ and $\nabla Z(p)=D^2u(p)(K_2)-DK_2(\nabla u)=D^2u(p)(\vec{e_{\theta}})-DK_2(\nabla u)=0$. The fact that $DK_2(\nabla u)=0$ follows by explicit computations: at $p$, $\nabla u=(\pm|\nabla u|,0,0)$, and  $DK_2(X)=(0,X_3,-X_2)$ for a vector $X=(X_1,X_2,X_3)$.
\item Assume that $\lambda\ne 0$. We claim that it is possible to find a linear combination $K=aK_1+bK_2$ such that $Z=\langle K,\nabla u\rangle$ satisfies $Z(p)=0$, $\nabla Z(p)=0$ and $-\Delta Z-f'(u)Z=0$ in $\Omega$. The first and third properties are satisfied for any choice of $a,b$. We can assume that $p=(0,1,0)$, so that $\nabla u=\pm|\nabla u|\vec{e_{\phi}}$. Now, we write
\begin{multline*}
\nabla Z=D^2u(p)(aK_1(p)+bK_2(p))-a DK_1(p)(\nabla u(p))-b DK_2(p)(\nabla u(p))\\
=(\pm|\nabla u|a+\lambda b)\vec{e_\theta},
\end{multline*}
where we have used the fact that at $p$, $K_1=0$ and $K_2=\vec{e_{\theta}}$. It is now sufficient to choose, e.g., $a=\lambda$, $b=\mp|\nabla u|$. The rest of the argument is the same as in Subsections \ref{nondegenerate} and \ref{zerosequal}.
\end{itemize}

\subsection{Application of Poincaré-Hopf and conclusion.} We have proved that $V$ and $\nabla u$ have the same zeros outside the equator, and that they are isolated since they are non-degenerate. Also the zeros of $V$ on the equator are non-degenerate, hence isolated. We have also proved that ${\rm Ind}_pV=1$ for any $p$ such that $V(p)=0$. Moreover $\langle V,\nu\rangle>0$ on $\partial\Omega$. We are in position to apply Poincaré-Hopf Theorem and conclude that
$$
1=\chi(\Omega)=\sum_{p:V(p)=0}{\rm Ind}_pV=\sum_{p:V(p)=0}1\geq \sum_{p:\nabla u(p)=0}1=\#\{{\rm critical\ points\ of\ }u\}\geq 1,
$$
therefore $V$ has a unique, non degenerate zero. It cannot be on the equator, since this would not be a zero of $\nabla u$, which must vanish at least once in $\Omega$. Therefore the unique zero of $V$ must coincide with a zero of $\nabla u$, which is therefore unique and non-degenerate (a maximum).

\begin{rem}
We have proved that, a posteriori, there are no points on $\Omega\cap\{\theta=\frac{\pi}{2}\}$ where $\langle\nabla u,\vec{e_{\theta}}\rangle=0$. 
\end{rem}

\subsection{The function $\langle K_1,\nabla u\rangle$ has exactly two zeros on $\partial\Omega$}\label{twozeros}We are only left to prove that condition \eqref{assumption} implies that $\langle K_1,\nabla u\rangle$ has exactly two zeros on $\partial\Omega$. We consider coordinates $(\theta,\phi)$ centered at the north pole $P$. Here $\theta$ is the geodesic distance from $P$.

We consider a parametrization $\gamma:[0,|\partial\Omega|]\to\mathbb S^2\subset\mathbb R^3$ of $\partial\Omega$
$$
\gamma(t):=(\sin(\theta(t))\cos(\phi(t)),\sin(\theta(t))\sin(\phi(t)),\cos(\theta(t))),
$$
 where $t$ is the arc-length parameter. Since $\partial\Omega$ is uniformly star-shaped  we have that
$$
\phi:[0,|\partial\Omega|]\to[0,2\pi]
$$
is a strictly increasing bijection: $\phi'(t)>0$ for all $t\in[0,|\partial\Omega|]$.

Since $u=0$ on $\partial\Omega$, we can write, with abuse of notation, that $u(\theta(t),\phi(t))=0$ for all $t\in[0,|\partial\Omega|]$. Hence $\theta'(t)\partial_{\theta}u+\phi'(t)\partial_{\phi}u=0$. We write explicitly $\langle K_1,\nabla u\rangle=\cos(\phi)\partial_{\theta}u-\cot(\theta)\sin(\phi)\partial_{\phi}u$.
Then, any $t\in[0,|\partial\Omega|]$ for which $\langle K_1,\nabla u\rangle|_{\partial\Omega}=0$ at $t$ is such that
$$
\begin{cases}
\theta'(t)\partial_{\theta}u+\phi'(t)\partial_{\phi}u=0\\
\cos(\phi(t))\partial_{\theta}u-\cot(\theta(t))\sin(\phi(t))\partial_{\phi}u=0.
\end{cases}
$$
Since $\nabla u\ne 0$ on $\partial\Omega$ by assumption, in correspondence of a solution  $t\in[0,|\partial\Omega|]$ of the above system, the determinant of the associated matrix must vanish. In other words, for such $t$ we must have
$$
F(t):=\theta'(t)\cot(\theta(t))\sin(\phi(t))+\phi'(t)\cos(\phi(t))=0.
$$
To simplify the notation, we shall omit from now on the dependence in $t$. We prove now that $F$ 
has exactly two zeros in $[0,|\partial\Omega|]$. This implies that the system cannot have more than two solutions. Actually, in view of the geometric interpretation of the problem, this implies that there are exactly two solutions, which correspond to the points where integral curves of $K_1$ are tangent to $\partial\Omega$.

We note that
$$
F(0)=\phi'(0)>0\,,\ \ \ F(\phi^{-1}(\pi))=-\phi'(\phi^{-1}(\pi))<0.
$$
Therefore $F$ is a smooth periodic function which admits at least two zeros. Let $t$ be such that $F(t)=0$. We compute
\begin{multline*}
\phi'(t)F'(t)\\=-\sin(\phi)\phi'^3-\frac{\sin(\phi)\phi'\theta'^2}{\sin^2(\theta)}+\cot(\theta)\cos(\phi)\phi'^2\theta'+\cot(\theta)\sin(\phi)\phi'\theta''+\cos(\phi)\phi'\phi''.
\end{multline*}
Now, for a $t$ such that $F(t)=0$ we have that $\phi'\cos(\phi)=-\cot(\theta)\sin(\phi)\theta'$. Substituting this in the previous formula, when $F(t)=0$, we have
$$
\phi'(t)F'(t)=-\sin(\phi)\left(\phi'\left(\phi'^2+\frac{1+\cos^2(\theta)}{\sin^2(\theta)}\theta'^2\right)+\cot(\theta)(\theta'\phi''-\phi'\theta'')\right)
$$
Assume that
\begin{equation}\label{der_sign}
\phi'\left(\phi'^2+\frac{1+\cos^2(\theta)}{\sin^2(\theta)}\theta'^2\right)+\cot(\theta)(\theta'\phi''-\phi'\theta'')\ne 0
\end{equation}
for all $t$ (in other words, it has constant sign and never vanishes). Since for $t=0$ and for $t=\phi^{-1}(\pi)$ $F$ does not vanish, we deduce that $F'(t)<0$ in correspondence of a zero of $F$ in $[0,\phi^{-1}(\pi)]$, while $F'(t)>0$ in correspondence of a zero of $F$ in $[\phi^{-1}(\pi),|\partial\Omega|]$. We conclude that $F$ has exactly one zero in $[0,\phi^{-1}(\pi)]$ and exactly one zero in $[\phi^{-1}(\pi),|\partial\Omega|]$.

\begin{rem} We remark that the argument is invariant under rotations of the domain around $P$, hence the statement remains valid if we replace $K_1$ by $K_2$, or by any other linear combination of $K_1,K_2$.
\end{rem}

We prove now that \eqref{assumption} implies \eqref{der_sign}.

We consider the above arc-lenght parametrization $\gamma$ of $\partial\Omega$. We have
\begin{multline*}
\gamma'(t)=(\cos(\theta(t))\cos(\phi(t))\theta'(t)-\sin(\theta(t))\sin(\phi(t))\phi'(t),\\
 \cos(\theta(t))\sin(\phi(t))\theta'(t)+\sin(\theta(t))\cos(\phi(t))\phi'(t)\\
 ,-\sin(\theta(t))\theta'(t))
\end{multline*}
and
$$
|\gamma'(t)|^2=\theta'(t)^2+\sin^2(\theta(t))\phi'(t)^2=1
$$ 

We fix an orientation for $\gamma$, in such a way that $(\gamma'(t),N(t),\nu(t))$ is positively oriented with respect to the canonical basis $(\vec{e_1},\vec{e_2},\vec{e_3})$ of $\mathbb R^3$. Here $N(t)=\gamma(t)$ is the normal vector to $\mathbb S^2$ at $\gamma(t)$ and $\nu(t)$ is the unit outer normal to  $\partial\Omega$.

With this choice we have
\begin{multline}\label{nu}
\nu(t)=\gamma'(t)\times N(t)=(\sin(\phi(t)\theta'(t)+\sin(\theta(t))\cos(\theta(t))\cos(\phi(t))\phi'(t),\\
-\cos(\phi(t)\theta'(t)+\sin(\theta(t))\cos(\theta(t))\sin(\phi(t))\phi'(t),\\
-\sin^2(\theta(t))\phi'(t)).
\end{multline}

Then, the geodesic curvature $\kappa$ of $\partial\Omega$ with respect to $\nu$ is given by
\begin{equation}\label{f1}
\kappa=-\langle\gamma''(t),\nu(t)\rangle,
\end{equation}
where the scalar product is the standard product of $\mathbb R^3$ (which coincides with the scalar product on $\mathbb S^2$ if restricted to tangent vectors).

We also have
\begin{equation}\label{f2}
\vec{e_{\theta}}=(\cos(\theta(t))\cos(\phi(t)),\cos(\theta(t))\sin(\phi(t)),-\sin(\theta(t))).
\end{equation}
Using \eqref{nu}, \eqref{f1} and \eqref{f2} we find that

\begin{multline}\label{p0}
\cos(\theta)\kappa+\sin(\theta)\langle\nu,\vec{e_{\theta}}\rangle\\
=\sin^2(\theta(t))\phi'(t)+2\cos^2(\theta(t))\theta'(t)^2\phi'(t)+\sin^2(\theta(t))\cos^2(\theta(t))\phi'(t)^3\\
+\sin(\theta(t))\cos(\theta(t))\left(\theta'(t)\phi''(t)-\phi'(t)\theta''(t)\right).
\end{multline}
We use now that $1=\theta'(t)^2+\sin^2(\theta(t))\phi'(t)^2$ to re-write the first term on the right-hand side of \eqref{p0} as $(\theta'(t)^2+\sin^2(\theta(t))\phi'(t)^2)\sin^2(\theta(t))\phi'(t)$. We use this identity in \eqref{p0} to finally obtain
\begin{multline}\label{p1}
\cos(\theta)\kappa+\sin(\theta)\langle \nu,\vec{e_{\theta}}\rangle\\
=\sin^2(\theta(t))\left(\phi'(t)\left(\phi'(t)^2+\frac{1+\cos^2(\theta(t))}{\sin^2(\theta(t))} \theta'(t)^2\right)+\cot(\theta(t))\left(\theta'(t)\phi''(t)-\phi'(t)\theta''(t)\right)\right).
\end{multline}
Therefore from \eqref{assumption} it follows that \eqref{der_sign} is always strictly positive and never vanishes, and we conclude that $\langle K_1,\nabla u\rangle$ has exactly two zeros on $\partial\Omega$.

\section{Proof of Theorem \ref{main2} (uniqueness of the critical point)}\label{proof2}
We will prove here Theorem \ref{main2}. Since the proof is analogous to that of Theorem \ref{main} we shall only highlight the main computational differences. We remark that in case of $\mathbb H^2$ we do not have the additional difficulty that the two {\it hyperbolic} Killing fields $K_1,K_2$ defined in Subsection \ref{fields} are linearly dependent or vanish at some points. This renders the proof significantly simpler.

Without loss of generality we can assume that $\Omega$ is star-shaped with respect to the origin in the Poincaré disk model. As for the spherical case, we have that the unit outer normal to $\partial\Omega$ can be written as $\nu=-\frac{\nabla u}{|\nabla u|}$, since $|\nabla u|\ne 0$ on $\partial\Omega$. Sometimes it will be  convenient to describe $\mathbb H^2$ in polar coordinates $(r,\phi)\in[0,+\infty)\times[0,2\pi]$ based at the origin, where $r$ is the geodesic distance from the origin. In these coordinates the metric takes the form $dr^2+\sinh^2(r)d\phi^2$.

\subsection{The auxiliary vector field $V$} The vector field $V$ used in the proof of Theorem \ref{main2} has the same form of \eqref{V0}, namely
$$
V=-\tilde V^{\perp}
$$
with
$$
\tilde V=\langle K_1,\nabla u\rangle\nabla(\langle K_2,\nabla u\rangle)-\langle K_2,\nabla u\rangle\nabla(\langle K_1,\nabla u\rangle)
$$
where now $K_1,K_2$ are the Killing fields in the hyperbolic space defined in Subsection \ref{fields}.

\subsection{Critical points of $u$ are non degenerate} The proof that the critical points of $u$ are non degenerate is identical to the corresponding one in Subsection \ref{nondegenerate} for the case in which $p$ does not belong to the equator.

 A fundamental ingredient in the proof is the fact that given $K$, an arbitrary combination of $K_1,K_2$, then $Z:=\langle K,\nabla u\rangle$ has exactly two zeros on $\partial\Omega$ if condition \eqref{assumption2} holds. We will prove this fact in Subsection \ref{twozeros2}.
\subsection{The zeros of $V$ coincide with the zeros of $\nabla u$} Since at any point of $\mathbb H^2$ the fields $K_1,K_2$ are linearly independent, the same argument of Subsection \ref{zerosequal} applies. We deduce that the critical points of $u$ coincide with the zeros of $V$.
\subsection{The vector field $V$ satisfies $\langle V,\nu\rangle>0$ on $\partial\Omega$} 

As in Subsection \ref{boundaryC}, we see that
$$
\langle V,\nu\rangle_{|_{\partial\Omega}}=-\frac{1}{|\nabla u|}\langle\tilde V,\nabla^{\perp}u\rangle
$$
and that
\begin{align}
-\langle\tilde V,\nabla^{\perp}u\rangle\nonumber\\
=&\langle K_1,\nabla u\rangle DK_2(\nabla u,\nabla^{\perp}u)-\langle K_2,\nabla u\rangle DK_1(\nabla u,\nabla^{\perp}u)\label{red2}\\
-&\langle K_1,\nabla u\rangle D^2u(K_2,\nabla^{\perp}u)+\langle K_2,\nabla u\rangle D^2u(K_1,\nabla^{\perp}u)\label{blue2}.
\end{align}
We consider first \eqref{blue2} and note that the same computations of Subsection \ref{boundaryC} apply. Then  \eqref{blue2} can be written as
$$
|\nabla u|^3A(K_1,K_2)\kappa
$$
where $\kappa$ is the geodesic curvature of $\partial\Omega$ and $A(K_1,K_2)$ is the signed area of the parallelogram of sides $K_1,K_2$, which is positive since the basis $(K_1,K_2)$ is positively oriented at any point of $\mathbb H^2$. An explicit computation shows that
$$
A(K_1,K_2)=\frac{1+x^2+y^2}{1-x^2-y^2}.
$$
Recalling that in polar coordinates based at the origin $x^2+y^2=\tanh(r/2)$, we also have
$$
A(K_1,K_2)=\frac{1+x^2+y^2}{1-x^2-y^2}=\cosh(r).
$$
On the other hand, following Subsection \ref{boundaryC}, we rewrite \eqref{red2} as
$$
|\nabla u|^2\langle F,\nabla u\rangle,
$$
where
$$
F=DK_2\left(\frac{K_1}{|K_1|},\frac{K_1^{\perp}}{|K_1|}\right)K_1-DK_1\left(\frac{K_2}{|K_2|},\frac{K_2^{\perp}}{|K_2|}\right)K_2,
$$
which is a combination of the two Killing fields $K_1,K_2$. First we observe that, due to the ant-symmetry of $DK_i$, we have, for any positively oriented orthonormal basis $(\vec{e_1},\vec{e_2})$ of $T_p\mathbb H^2$, that (at $p$) $DK_i\left(\frac{K_j}{|K_j|},\frac{K_j^{\perp}}{|K_j|}\right)=DK_i(\vec{e_1},\vec{e_2})$, for any $i,j\in\{1,2\}$. Then, thanks to explicit computations of $DK_i$ using the Christoffel  symbols for the Poincaré disk model, we get that $F=x\partial_x+y\partial_y$. Therefore, recalling that $x\partial_x+y\partial y$ has norm $\frac{2\sqrt{x^2+y^2}}{1-x^2-y^2}$ and that in polar coordinates $x^2+y^2=\tanh(r/2)$ we get that \eqref{red2} can be written as
$$
-|\nabla u|^3\sinh(r)\langle\nu,\vec{e_r}\rangle.
$$
We conclude that under the condition \eqref{assumption2} we have that $\langle V,\nu\rangle_{|_{\partial\Omega}}=|\nabla u|^2(\cosh(r)\kappa-\sinh(r)\langle\nu,\vec{e_r}\rangle)>0$.

\subsection{The index of any zero of $V$ is $1$}\label{indexS2} The computations are exactly the same of the corresponding case of Subsection \ref{indexS} when $p$ does not belong to the equator.

\subsection{Application of Poincaré-Hopf and conclusion} We have proved that $V$ and $\nabla u$ have the same zeros, which are isolated since they are non-degenerate. Moreover any such zero has index $1$ for $V$. Therefore we can apply Poincaré-Hopf Theorem and conclude:
$$
1=\chi(\Omega)=\sum_{p:V(p)=0}{\rm Ind}_pV=\sum_{p:V(p)=0}1=\#\{{\rm critical\ points\ of\ u}\}.
$$
Hence $u$ has a unique, non-degenerate critical point (a maximum).

\subsection{The function $\langle K_1,\nabla u\rangle$ has exactly two zeros on $\partial\Omega$}\label{twozeros2} We prove here that assumption \eqref{assumption2} implies that $\langle K_1,\nabla u\rangle$ has exactly two zeros on $\partial\Omega$. The proof is very similar to that presented in Subsection \ref{twozeros}, however we will give a few details. 

First, we note that being the metric conformal to the Euclidean one, it is sufficient to prove the statement when we consider the Euclidean metric on the unit disk $D$ instead of the metric $\frac{4}{(1-x^2-y^2)^2}(dx^2+dy^2)$. Recall that we assumed $\Omega$ to be uniformly star-shaped with respect to the origin.

We consider then an arc-length parametrization $\gamma:[0,|\partial\Omega|]\to D$ of $\partial\Omega$, $\gamma(t)=(x(t),y(t))\in D$. The outward unit normal to $\partial\Omega$ is given by $\nu_E=(y'(t),-x'(t))$. We use the subscript $\nu_E$ since this is the unit outer normal with respect to the Euclidean metric. We also denote by $\nabla_Eu$ the Euclidean gradient and by $\langle\cdot,\cdot\rangle_E$ the Euclidean scalar product. Let now $t\in[0,|\partial\Omega|]$ be such that $\langle K_1,\nabla_Eu\rangle_E$=0. Then
$$
\begin{cases}
x'(t)\partial_xu+y'(t)\partial_yu=0\\
\frac{1-x^2(t)+y^2(t)}{2}\partial_xu-x(t)y(t)\partial_yu=0.
\end{cases}
$$
Since $\nabla u\ne 0$ on $\partial\Omega$, in correspondence of a solution $t\in[0,|\partial\Omega|]$ of the system, the associated determinant must vanish, i.e., for such $t$ we must have
$$
F(t):=x'(t)x(t)y(t)+y'(t)\left(\frac{1-x^2(t)+y^2(t)}{2}\right)=0.
$$
From now on we will drop the dependence on $t$ in order to simplify the notation. We prove that $F$ has exactly two zeros in $[0,|\partial\Omega|]$. This implies that the system has no more than two solutions, which must be exactly two and correspond to the points where integral curves of $K_1$ are tangent to $\partial\Omega$.

First, we note that since $\Omega$ is uniformly star-shaped, we have $xy'-x'y\geq\delta\geq 0$ for some positive constant $\delta$. We can assume that $y(0)=0$ and $x(0)>0$; moreover, there exists a unique $t^*$ such that $y(t^*)=0$ , $x(t^*)<0$ and $y>0$ on $(0,t^*)$. Therefore $y'(0)>0$ and $y'(t^*)<0$, which in turn implies that $F(0)>0$ and $F(t^*)<0$. 
We prove now that $F$ has a unique zero in $(0,t^*)$.  To do so, we compute $F'$:
$$
F'=y+xyx''+y''\left(\frac{1-x^2+y^2}{2}\right).
$$
Let now $t\in(0,t^*)$ be a point where $F(t)=0$. Therefore $x'(t)\ne 0$ otherwise we would have $y'(t)=0$ which is not possible. Moreover, at such $t$, we must have $x'(t)<0$; if this were not true, from the uniform star-shapedness we would have $x(t)y'(t)>0$, but then $F(t)\ne 0$. Let us multiply $F'$ by $x'$ and use the fact that $x'xy=-y'\left(\frac{1-x^2+y^2}{2}\right)$. At such $t$ we have
\begin{multline}\label{F1}
x'(t)F'(t)\\=x'y+x'xyx''+x'y''\left(\frac{1-x^2+y^2}{2}\right)
=\left(\frac{1-x^2+y^2}{2}\right)(x'y''-y'x'')+x'y\\
=\frac{1-x^2+y^2}{1+x^2+y^2}\left(\frac{1+x^2+y^2}{2}(x'y''-y'x'')+\left(\frac{1+x^2+y^2}{1-x^2+y^2}\right)x'y\right)\\
=\frac{1-x^2+y^2}{1+x^2+y^2}\left(\frac{1+x^2+y^2}{2}(x'y''-y'x'')-xy'+x'y\right)\\
=\frac{1-x^2+y^2}{1+x^2+y^2}\left(\frac{1+x^2+y^2}{2}\kappa_E-\langle\nu_E,(x,y)\rangle_E\right).
\end{multline}
Here we have denoted by $\kappa_E$ the geodesic curvature of $\partial\Omega$ with respect to the Euclidean metric. In order to pass from the second to the third line we have used the condition $F(t)=x'xy+y'\left(\frac{1-x^2+y^2}{2}\right)=0$. If the last term in \eqref{F1} never vanishes in $(0,t^*)$ and is positive we are done, since this implies that $F$ has a unique zero in $(0,t^*)$, as in Subsection \ref{twozeros}. In the same way one can prove that $F$ has exactly one zero in $(t^*,|\partial\Omega|)$.

In order to conclude, we must prove that condition \eqref{assumption2} implies that the last term of $\eqref{F1}$ is positive an never vanishes on $(0,t^*)$. To do so, we recall that the geodesic curvature $\kappa$ of $\partial\Omega$ and the curvature $\kappa_E$ of $\partial\Omega$ for the Euclidean metric are related by
$$
\kappa=\frac{1-x^2-y^2}{2}\kappa_E+\langle\nu_E,(x,y)\rangle_E.
$$
Hence, we can write
\begin{multline}\label{equiv2}
\cosh(r)\kappa+\sinh(r)\langle\nu,\vec{e_r}\rangle=\frac{1+x^2+y^2}{1-x^2-y^2}\kappa+\frac{2}{1-x^2-y^2}\langle\nu_E,(x,y)\rangle_E\\=\frac{1+x^2+y^2}{2}\kappa_E-\langle\nu_E,(x,y)\rangle_E.
\end{multline}
The proof of the claim is now completed.

\section{ Optimality of the conditions \eqref{assumption} and  \eqref{assumption2} in Theorems \ref{main} and \ref{main2}}\label{s10}
In this section, we focus on the case of the torsion problem
\begin{equation}\label{T}
\begin{cases}
-\Delta u = 1 &
\mbox{ on } \Omega \\
u = 0 & \mbox{ on } \partial \Omega.
\end{cases}
\end{equation}
Our goal is to construct examples of solutions to \eqref{T} in certain geodesically star-shaped domains of $M = \mathbb{S}^2$ or $\mathbb{H}^2$ that possess as many critical points as desired, even when the domain $\Omega$, in a suitable sense,  almost satisfies \eqref{assumption} or \eqref{assumption2}. 

The examples we construct demonstrate that the conditions \eqref{assumption} and \eqref{assumption2} in Theorems \ref{main} and \ref{main2}, which guarantee the uniqueness of the critical point of $u$, are in fact optimal.\

This type of example first appeared in \cite{GG22} and \cite{DRG1}, where the authors showed the optimality of the convexity assumption on $\Omega$ in the result of \cite{CC} when $\Omega \subset \mathbb{R}^2$. The example in \cite{GG22} was later generalized in \cite{EGG25} to the case of a general $d$-dimensional Riemannian manifold $M$, as well as to problems involving general nonlinearities $f(u)$. 

The resulting example is a smooth domain which collapses to a point but does not fully inherit the intrinsic features of the manifold $M$.

Here, we aim to produce more constructive examples of solutions to \eqref{T}, which retain all the desirable properties of those in \cite{GG22} and \cite{EGG25}, by working directly in $\mathbb{S}^2$ and $\mathbb{H}^2$ and which do not shrink to a point. Since all these examples are, in some sense, pathological, we will observe that the underlying domain converges to a suitable geodesic segment. This kind of pathology in the domain is necessary, as shown in \cite{BDM} for the Euclidean case: a small perturbation of a convex domain does not necessarily produces additional critical points.

Through all this section we will consider coordinates $(x,y)\in\mathbb R^2$ for both $\mathbb S^2$ and $\mathbb H^2$: for $\mathbb S^2$, $(x,y)\in\mathbb R^2$ are the stereographic coordinates obtained through the stereographic projection from the south pole, while for $\mathbb H^2$, $(x,y)\in D\subset\mathbb R^2$, where $D$ is the unit disk, are the coordinates in the Poincaré disk model.
\begin{thm}\label{T1}
For any integer $n\geq 2$ there exists a family of smooth bounded domains $\Omega_b$ either in $\mathbb S^2$ or in $\mathbb H^2$
 and smooth functions $u_b$ which solve \eqref{T} in $\Omega_{b}$ such that, for all $b$ small enough:
\begin{enumerate}[a)]
\item the solution $u_b$ has at least $n$ maximum points in $\Omega_{b}$;
\item we have $\overline\Omega_{b}\to [-\bar x_{\eta},\bar x_{\eta}]\times\{0\}$ as $b\to 0$ (an arc of geodesic), where $\bar x_{\eta}\in (0,1)$;
\item $\Omega_{b}$ is geodesically starshaped with respect to the origin (the projection of the norh pole in the $\mathbb S^2$ case); 
\item for $b$ small enough we have
\begin{equation}\label{cla-f}
\begin{split}
\partial\Omega_{b}& =\left\{(x,y):-\bar x_{\eta}\le x\le\bar x_{\eta}\,, y=\pm\sqrt{\frac{b}{2}(1-\eta f(x))}(1+x^2)(1+o(1))\right\}\subset \mathbb S^2,\\
\partial\Omega_{b}&=\left\{(x,y):-\bar x_{\eta}\le x\le\bar x_{\eta}\,, y=\pm\sqrt{\frac{b}{2}(1-\eta f(x))}(1-x^2)(1+o(1))\right\}\subset \mathbb H^2
\end{split}
\end{equation}
where $\eta>0$ is a suitable constant and $f$ is a polynomial of degree $2n$;
\item (sharpness of  \eqref{assumption} and \eqref{assumption2}). We have that
\begin{equation}\label{cla}
\lim\limits_{b\to0}(\cos(\theta)k_{\Omega_b}+\sin(\theta)\langle\nu,\vec{e_{\theta}}\rangle)\ge0\,\ \ \hbox{ on }\partial\Omega_b,\subset \mathbb S^2,
\end{equation}
\begin{equation}\label{cla1}
\lim\limits_{b\to0}(\cosh(r)k_{\Omega_b}-\sinh(r)\langle\nu,\vec{e_r}\rangle)\ge0\,\ \ \hbox{ on }\partial\Omega_b\subset \mathbb H^2,
\end{equation}

where $\theta={\rm dist}(\cdot,0)$ is the geodesic distance from the origin (i.e., from the north pole) in spherical coordinates and $r={\rm dist}(\cdot,0)$ is the geodesic distance from the origin in $ \mathbb H^2$.
\end{enumerate}
\end{thm}
The two examples are built in a similar way: in both cases we start with a solution of the torsion problem \eqref{T} in a {\it strip} around a geodesic (i.e., the set of points at distance at most $c>0$ from the geodesic). In both cases, the solution is explicit and has a continuum of critical points: the whole geodesic. 
Then we perturb the solution adding an harmonic function $bv$ where $b$ is a positive parameter which will go to zero and $v$ is an harmonic polynomial of degree $2n$ with zeros satisfying some suitable assumptions.
 
It is interesting to note that, taking a geodesic $\gamma$ in $\mathbb S^2$ or $\mathbb H^2$, a point $P\in\gamma$, a strip $S=\{x:{\rm dist(x,\gamma)}<c\}$, $c>0$, and $r={\rm dist(\cdot,P)}$, then it is standard to check that the right-hand sides of \eqref{assumption} and \eqref{assumption2} are identically zero on  $\partial S$, giving already an hint of the fact that the conditions are optimal.

\subsection{Proof of Theorem \ref{T1}}\label{s7}
We will consider the following two models:\\   $\left(\mathbb R^2,\frac{4}{(1+x^2+y^2)^2}(dx^2+dy^2)\right)$ for $\mathbb S^2$ and $\left(D,\frac{4}{(1-x^2-y^2)^2}(dx^2+dy^2)\right)$ for $\mathbb H^2$, where $D=\{(x,y):x^2+y^2<1\}$ is the unit disk in $\mathbb R^2$. The model for the Hyperbolic plane is just the Poincaré Disk model, while the standard metric for the sphere is written in stereographic coordinates. If the stereographic projection is from the south pole, then $D$ corresponds to the northern hemisphere, while $\mathbb R^2\setminus D$ corresponds to the southern hemisphere.

Let  $b>0$ and $\gamma:=\{(x,y):y=0\}$. Define

\begin{equation}\label{sfb}
\mathcal S_b:=\{p\in M: {\rm dist}(p,\gamma))<d_M(\sqrt{\tanh(b/2)})\}.
\end{equation}
Here ${\rm dist}$ is the Riemannian distance and $d_M:[0,+\infty)\to[0,+\infty)$ is defined by
$$
d_M(\cdot)=\begin{cases}
2\arctan(\cdot )\,, & {\rm when\ }M=\mathbb S^2\\
2\arctanh(\cdot)\,, & {\rm when\ }M=\mathbb H^2.
\end{cases}
$$
In other words, $\mathcal S_b$ is the sets of points at distance at most $d_M(\sqrt{\tanh(b/2)})$ from $\gamma$. Note that in both models the curve $\gamma$ is a geodesic: in the Poincaré disk model, it is the segment $(-1,1)\times\{0\}$, while it corresponds to an arc of great circle passing through the poles in the spherical case. See Figure \ref{strips}. In particular, in the spherical case, $\mathcal S_b$ is a band around a great circle, see Figure \ref{strips_sphere}.

\begin{figure}
\begin{center}
  \includegraphics[width=\textwidth]{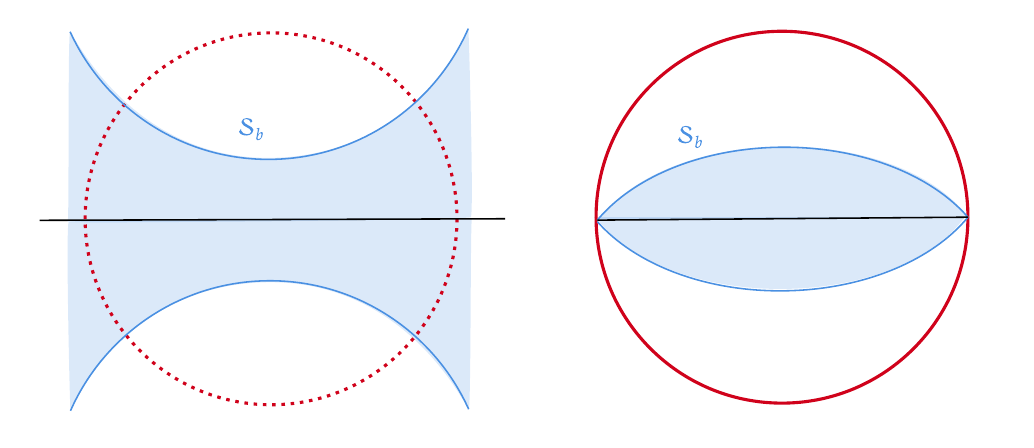}
  \caption{The strip $\mathcal S_b$ in the stereographic projection of $\mathbb S^2$ (left) and in the Poincaré disk model of $\mathbb H^2$ (right).}
  \label{strips}
\end{center}
\end{figure}

\begin{figure}
\begin{center}
  \includegraphics[width=\textwidth]{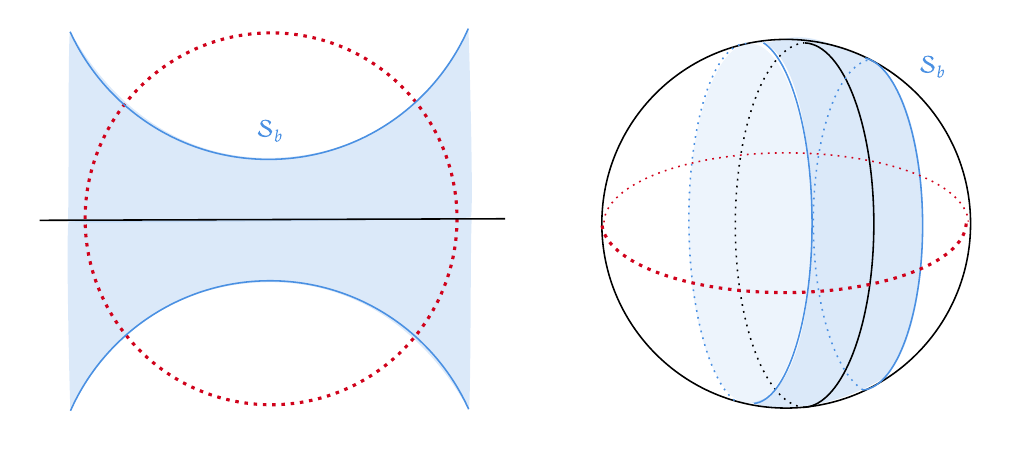}
  \caption{In the stereographic projection (from the south pole) of the sphere, the origin corresponds to the north pole, the $x$ axis to a great circle passing through the north pole, the unit circle (dotted red circle) corresponds to the equator.}
  \label{strips_sphere}
\end{center}
\end{figure}

A more explicit way of describing $\mathcal S_b$ is the following
\begin{equation}\label{sfb_2}
\mathcal S_b=\begin{cases}
(x,y) \in \R^2: 4y^2 < (1+x^2+y^2)^2(1-e^{-2b})\,, & M=\mathbb S^2\\
(x,y) \in D: 4y^2<(1-x^2-y^2)^2(e^{2b}-1)\,, & M=\mathbb H^2.
\end{cases}
\end{equation}
In the stereographic projection of $\mathbb S^2$, the strip is represented by the intersection of the exteriors of two disks, while in the Poincaré disk model of $\mathbb H^2$, the strip is represented by the intersection of two disks. See Figure \ref{strips}. 

\medskip

In this section with  $\Delta$ we denote the Laplacian for the metric we are considering (that is, the standard metric of $\mathbb S^2$ or $\mathbb H^2$), while $\Delta_{\mathbb R^2}$ is the Laplacian for the usual Euclidean metric, namely $ \Delta_{\mathbb R^2}=\partial^2_{xx}+\partial^2_{yy}$.
 We consider a solution $\psi_b$ of the torsion problem \eqref{T} in $\mathcal S_b$. Namely $\psi_b$ satisfies 
\begin{equation}\label{ma40}
\begin{cases}
                        -\Delta \psi_b=1  &
            \mbox{  on }\mathcal S_b\\
\psi=0  & \mbox{  on }\partial \mathcal S_b
\end{cases}
\end{equation}
which is equivalent to
\begin{equation}\label{ma40a}
\begin{cases}
                        -\Delta_{\R^2} \psi_b=\frac4{\left(1+x^2+y^2\right)^2}
            & {\rm when\ }S_b\subset\mathbb S^2,\\\
u=0  & \mbox{  on }\partial \mathcal S_b
\end{cases}
\end{equation}
and
\begin{equation}\label{ma40b}
\begin{cases}
                        -\Delta_{\R^2} \psi_b=\frac4{\left(1-x^2-y^2\right)^2}
            & {\rm when\ }S_b\subset\mathbb H^2,\\\
u=0  & \mbox{  on }\partial \mathcal S_b.
\end{cases}
\end{equation}
It is easily seen that
\begin{equation}\label{psib}
\psi_b(x,y)=
\begin{cases}
\frac12\log\left(1-\frac{4y^2}{(1+x^2+y^2)^2}\right)+b\,, {\rm when\ }M=\mathbb S^2,\\
-\frac12\log \left(1+\frac{4y^2}{(1-x^2-y^2)^2}\right)+b\,, {\rm when\ }M=\mathbb H^2,
\end{cases}
\end{equation}
is the unique solution to (\ref{ma40})  in the strip $\mathcal S_b$.

Moreover $\nabla \psi_b(x,y)=0 \iff y=0\iff (x,y)\in \gamma$.  

To obtain a function which verifies the properties of Theorem \ref{s7} in $\Omega_b$, we perturb $\psi_b$ adding a harmonic polynomial of degree $2n$ with $n>1$,  that we denote by $-b\eta v(x,y)$, where $\eta>0$ is a suitable parameter. Of course the perturbed function 
$$u_b=\psi_b-b\eta v$$ 
still verifies the torsion equation $-\Delta u_b=1$ and $\Omega_b$ will be given by the zero superlevel set of this function, see Figures \ref{fig4} and \ref{figIp}. The advantage of working with metrics that are conformal to the Euclidean one, is that we can look for the functions $v$ among harmonic functions for the Euclidean metric, and these functions will be used for both models. 
\begin{figure}
\includegraphics[width=\textwidth]{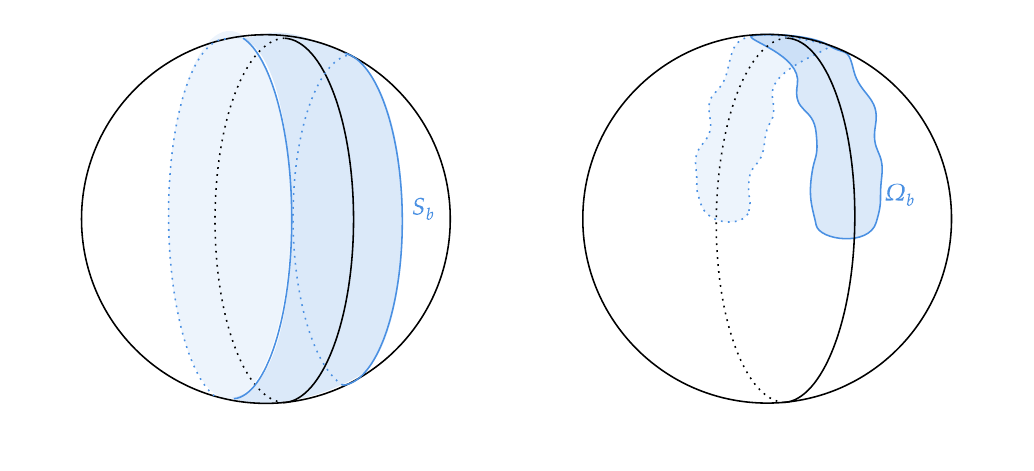}
\caption{The unperturbed domain $\mathcal S_b$ and the final domain $\Omega_{b}$ in the spherical case.}
\label{fig4}
\end{figure}
\begin{figure}
\begin{center}
  \includegraphics[width=\textwidth]{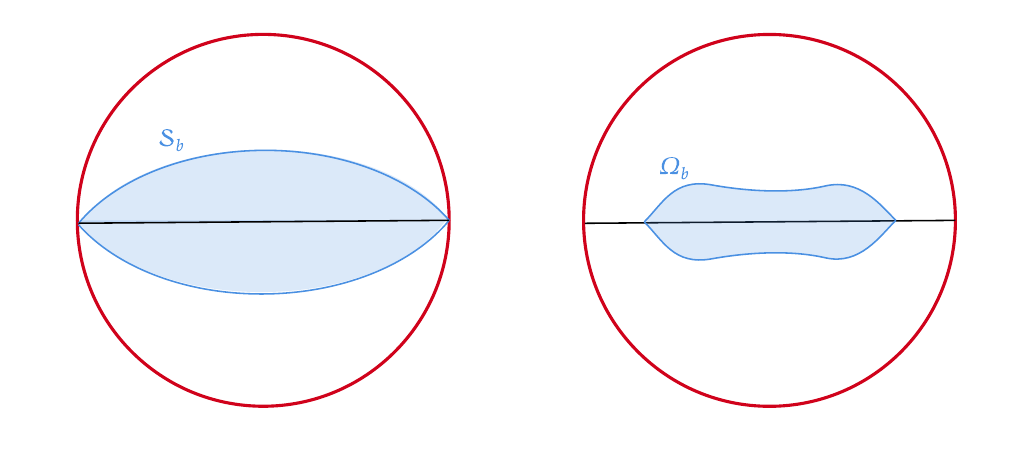}
  \caption{The unperturbed domain $\mathcal S_b$ and the final domain $\Omega_b$ in the hyperbolic case. }
    \label{figIp}
\end{center}
\end{figure}
\medskip

Next we introduce the function $v$.
Fix $n>1$, take $n$ real numbers $0<a_1<a_2<..<a_n<1$ and  define the real polynomial
\begin{equation}\label{max10}
f(x)=\Pi_{i=1}^n\left(x^2-a_i^2\right)
\end{equation}
so that $f(x)$ satisfies the following condition:
\begin{equation}\label{fg}
\begin{cases}
f(1)> \sup _{x\in (-a_n,a_n)}
\left[f(x)-\frac {x(1+x^2)}{2(1-x^2)}f'(x)\right] & {\rm when\ }M=\mathbb S^2\\
f(1)> \sup _{x\in (-a_n,a_n)}
\left[f(x)-\frac {x(1-x^2)}{2(1+x^2)}f'(x)\right]&{\rm when\ }M=\mathbb H^2.
\end{cases}
\end{equation}
Let us make some remarks on the properties of $f$ which will be used later.
The polynomial $f(x)$ is even in $x$, it has $2n$ zeroes $\pm a_i$ for $i=1,..,n$,  it has $n$ minimum points $x_1,..,x_n\in (-a_n,a_n)$
and  $n-1$ maximum points $\tilde x_1,..,\tilde x_{n-1}\in (-a_n,a_n)$; all the minima are strictly negative and all the maxima are strictly positive.  

Moreover $f$ is strictly monotone in $|x|\geq a_n$ and for every $c\in (\sup _{x\in (-a_n,a_n)}f(x),f(1))$ there exists a unique $\bar x_c\in (0,1)$ such that $f(\bar x_c)=c.$ Finally  it satisfies 
\begin{equation}\label{ma32}
xf'(x)\ge0\quad\hbox{for any }a_n\le|x|.
\end{equation}
\begin{remark}
To prove Theorem \ref{T1} we need a polynomial that satisfies \eqref{fg} for every $n>1$.  A sufficient condition on the points $a_1,..,a_n$ is when the number {$a_n$ is small}. 
Indeed is easily seen that
$$\sup _{x\in (-a_n,a_n)}f(x),f'(x)\to0\quad\hbox{as }a_n\to0$$
and
$$\sup _{x\in (-a_n,a_n)}f(1)\to1\quad\hbox{as }a_n\to0.$$
Hence $\eqref{fg}$ holds for every $n$ when $a_1,...,a_n$ are small enough.
Nevertheless the fact that $a_n$ is small is not necessary for \eqref{fg} to hold as one can see with the following example
 for  $n=3$,  where $f(x)=(x^2-\frac 14)(x^2-\frac 1{16})(x^2-\frac 9{16})$.
\end{remark}
Now we define 
\begin{equation}\label{max10b}
\boxed{v(x,y):=Re\left( f(z)\right)=Re\big( \Pi_{i=1}^n\left(z^2-a_i^2\right)\big)}
\end{equation}
where $f(z)$ is the extension of $f(x)$ to the complex plane.
It is easy to see that $v(x,y)$ is an harmonic polynomial of degree $2n$ that depends only on $x^2$ and $y^2$.\\
We take 
\beq\label{eta}
\eta=\frac 12\left[\frac 1{f(1)}+\frac 1{\sup _{x\in (-a_n,a_n)}
\left[f(x)-\frac {x(1+x^2)}{2(1-x^2)}f'(x)\right]
}
\right]
\eeq
in the spherical case and 
\beq\label{eta2}
\eta=\frac 12\left[\frac 1{f(1)}+\frac 1{\sup _{x\in (-a_n,a_n)}
\left[f(x)-\frac {x(1-x^2)}{2(1+x^2)}f'(x)\right]
}
\right]
\eeq
in the hyperbolic case\\
 (alternatively we can choose $\eta \in \left(\frac 1{f(1)},\frac 1{\sup _{x\in (-a_n,a_n)}
\left[f(x)-\frac {x(1+x^2)}{2(1-x^2)}f'(x)\right]}
\right)$  in $\mathbb S^2$ and analogously in $\mathbb H^2$).\\
Of course, due to \eqref{fg} 
$\frac 1{\eta}\in \left(\sup _{x\in (-a_n,a_n)}
\left[f(x)-\frac {x(1+x^2)}{2(1-x^2)}f'(x)\right], f(1)\right)$. We define 
$$\boxed{u_{b}(x,y)=\psi_b(x,y)-b \eta v(x,y)}$$
where $\psi_b$ is as in \eqref{psib}, namely
\begin{equation}\label{mam-f}
 u_{b}(x,y)=
\begin{cases}
\frac12\log\left(1-\frac{4y^2}{(1+x^2+y^2)^2}\right)+b\big(1-\eta v(x,y)\big)\,, {\rm when\ }M=\mathbb S^2,\\
-\frac12\log \left(1+\frac{4y^2}{(1-x^2-y^2)^2}\right)+b\big(1-\eta v(x,y)\big)\,, {\rm when\ }M=\mathbb H^2,
\end{cases}
\end{equation}
and we set 
\beq\label{f3}
\boxed{
\Omega_{b}:={\rm connected\ component\ of\ }\{(x,y):u_{b}>0\}{\rm\ containing\ }(0,0).}
\eeq
In the following, we will always use the symbol $\Omega_b$ to denote the domain both in the spherical and in the hyperbolic case. Of course $u_b$ solves \eqref{ma40a} and \eqref{ma40b} in $\Omega_{b}$. 
Our aim is to show that when $b$ is small the domain $\Omega_b$ satisfies the properties of Theorem \ref{T1}.

 We start by showing some properties of the domain $\Omega_b$.
 \begin{lemma}\label{l2b}
Let $u_b$ be as in \eqref{mam-f} and $\Omega_b$ be as in \eqref{f3}. The set $\Omega_b$ is not empty and contains the segment $\{(x,0): -\bar x_\eta<x<\bar x_\eta\}$, where $\bar x_\eta\in (a_n,1)$ is the unique positive solution to $f(x)=\frac 1{\eta}$. Moreover the points $(\pm \bar x_\eta,0)\in \partial \Omega _b$ for every $b$ and are the unique points of $\partial \Omega_b\cap \{y=0\}$. Further, when $b$ is small enough,  $\Omega_b$ is bounded and contained in $(-\bar x_\eta,\bar x_\eta)\times (-\delta,\delta)$,  with $\delta\in (0,1)$, and 
\begin{equation}\label{ma3}
\overline \Omega_b\to \{(x,0): -\bar x_\eta\leq x\leq \bar x_\eta\}\hbox{ as }b\to 0.
\end{equation}

\end{lemma}
\begin{proof}
First we compute $u_b(0,0)=b-b\eta v(0,0)=b(1-\eta f(0))>0$ since we have that $f(0)\le \sup _{x\in (-a_n,a_n)}f(x)<\frac 1 \eta$ by the definition of $\eta$ and the properties of $f(x)$. This shows that $\Omega_b$ is not empty for every $b$.\\
Now we compute $u_b$ along the $x$-axis and we have
\begin{equation}
u_b(x,0)=b\big(1-\eta v(x,0)\big)=b\big(1-\eta f(x)\big).
\end{equation}
Then \eqref{eta}, the definition of $\bar x_\eta$ and the properties of $f(x)$ implies that $f(\bar x_\eta)=\frac 1{\eta}$,  $f(x)<\frac1\eta$ for $|x|<\bar x_\eta$ and $f(x)>\frac1\eta$ for $|x|>\bar x_\eta$. Hence the segment $\{(x,0): -\bar x_\eta<x<\bar x_\eta\}\subset \Omega_b$ for every $b$ and the points 
$(\pm \bar x_\eta,0)$ are the unique points of $\partial \Omega_b\cap \{y=0\}$.
Now we show that when $b$ is small $\Omega_b\subset D$ in the spherical case since in the hyperbolic case is already verified.   Let us compute $u_b(x,y)$ along the unit circle $x^2+y^2=1$. When $y=0$ 
$$u_b(\pm 1,0)=b(1-\eta v(\pm 1,0))=b(1-\eta f(1))<0 $$
by the definition of $\eta$.  By the continuity of $v(x,y)$ there exists $\delta>0$ such that $v(\pm \sqrt{1-y^2},y)<\frac 1\eta$ for $|y|<\delta$ so that $u_b(\pm \sqrt{1-y^2},y)<0$ for $|y|<\delta$.  For $|y|\geq \delta$ instead 
\[\begin{split}
u_b( \pm \sqrt{1-y^2},y)&=\frac 12 \log\left(1-2y^2\right)+b(1-\eta v(\pm \sqrt{1-y^2},y))\\
&\leq \frac 12 \log\left(1-2\delta^2\right)+b(1+\eta \sup_{|y|\leq 1} v(\pm \sqrt{1-y^2},y))<0
\end{split}\]
if $b$ satisfies 
\[b<\frac{-\frac 12 \log\left(1-2\delta^2\right)}{1+\eta \sup_{|y|\leq 1} v(\pm \sqrt{1-y^2},y))}.\]
This proves that $u_b(x,y)$ is negative on the unit circle, for $b$ small enough, and shows that $\Omega_b\subset D$ in the spherical case.

Now we prove that $\Omega_b\subset \{(x,y)\in D: -\bar x_\eta<x<\bar x_\eta\}.$
Let us suppose by contradiction that there exists a sequence $b_k$, such that $b_k\to 0$ as $k\to \infty$ and points $y_k\in (-1,1)$ such that $u_{b_k}(\bar x_\eta,y_k)>0$.  Since $|b(1-\eta v(\bar x_\eta,y_k))|\leq Cb_k\to 0$ and since $\frac 12 \log \left(1-\frac{4y^2}{(1+\bar x_\eta^2+y^2)^2}\right)<0$ for $|y|> 0$ then, up to a subsequence,  $y_k\to 0$ as $k\to \infty$. Then
\[u_{b_k}(\bar x_\eta,y_k)=u_{b_k}(\bar x_\eta,0)+y_k \frac{\partial u_{b_k}}{\partial y}(\bar x_\eta,0)+\frac 12 y^2_k \frac{\partial^2 u_{b_k}}{\partial y^2}(\bar x_\eta,0)+o(y_k^2)\]
as $k\to \infty$. Moreover $u_{b_k}(\bar x_\eta,0)=0$ by the definition of $\bar x_\eta$ and $\frac{\partial u_{b_k}}{\partial y}(\bar x_\eta,0)=0$ since $u_b(x,y)$ depends on $y^2$.  Further
$\frac{\partial^2 u_{b_k}}{\partial y^2}(\bar x_\eta,0)=-\frac 4{(1+\bar x_\eta^2)}-b_k\eta \frac{\partial^2 v}{\partial y^2}(\bar x_\eta,0)$ and, since $v(x,y)$ is a polynomial $|\frac{\partial^2 v}{\partial y^2}(\bar x_\eta,0)|\leq C$
so that 
$\frac{\partial^2 u_{b_k}}{\partial y^2}(\bar x_\eta,0)<-\frac 2{(1+\bar x_\eta^2)}<0$ when $b<\bar b$ and $\bar b$ depends only on $\bar x_\eta$ and $\eta \frac{\partial^2 v}{\partial y^2}(\bar x_\eta,0)$.  This also shows that when $k$ is large and $b_k<\bar b$ then 
\[u_{b_k}(\bar x_\eta,y_k)<y_k^2\left(-\frac 2{(1+\bar x_\eta^2)}+o(y_k)\right)<0\]
and gives a contradiction with $u_{b_k}(\bar x_\eta,y_k)>0$. This proves that $\Omega_b\subset \{(x,y)\in D: -\bar x_\eta<x<\bar x_\eta\}$ in the spherical case.  
A very similar proof allows to have the same result in the hyperbolic case. 

Finally \eqref{ma3} is an easy consequence of the fact that  $|v(x,y)|\leq C$ in $\Omega_b$. Indeed
$$
u_b(x,y)\to \begin{cases}\frac 12 \log \left(1-\frac{4y^2}{(1+x^2+y^2)^2}\right)\leq 0\,, & {\rm when\ }M=\mathbb S^2,\\
-\frac 12 \log \left(1+\frac{4y^2}{(1-x^2-y^2)^2}\right)\leq 0\,, & {\rm when\ }M=\mathbb H^2,
\end{cases}
$$
and the convergence is uniform on compact subsets of $M$. 
\end{proof}
Next lemma will be used several times. It claims that $ \partial \Omega_b$ is ``almost'' a graph.
\begin{lemma}\label{ma11}
Let $(x_b,y_b)\in \partial \Omega_b$. If,  as $b\to 0$ $x_b\to \bar x\in [-\bar x_\eta,\bar x_\eta]$, then
{\small
\begin{equation}\label{ma6}
y_b=\begin{cases}
\pm\sqrt{\frac{1}2b\big(1-\eta f(\bar x)-\eta f'(\bar x)(x_b-\bar x)\big)}(1+\bar x^2)(1+O(y_b^2,|x_b-\bar x|^2,\frac{y_b^2}b|x_b-\bar x|))\,, & {\rm when\ }M=\mathbb S^2\\
\pm\sqrt{\frac{1}2b\big(1-\eta f(\bar x)-\eta f'(\bar x)(x_b-\bar x)\big)}(1-\bar x^2)(1+O(y_b^2,|x_b-\bar x|^2,\frac{y_b^2}b|x_b-\bar x|))\,, & {\rm when\ }M=\mathbb H^2.
\end{cases}
\end{equation}}
\end{lemma}
\begin{proof}
We prove the result in the case $M=\mathbb S^2$, the other case is analogous. A point $(x_b,y_b)\in\partial\Omega_b$ satisfies
\begin{equation}\label{f56}
\frac 12 \log \left(1-\frac{4y_b^2}{(1+x_b^2+y_b^2)^2}\right)=b\left(\eta v(x_b,y_b)-1\right).
\end{equation}
Since $|x_b|,|y_b|\leq C$ the right-hand side of \eqref{f56} is bounded by $Cb$ and goes to zero as $b\to 0$. Then,  if $x_b\to \bar x\in[-\bar x_\eta,\bar x_\eta]$,  $y_b\to 0$. We can then expand the left-hand side  and the right-hand side of \eqref{f56} and we get
\[\eta v(x_b,y_b)-1=\eta v(\bar x,0)-1
+\eta \frac{\partial v}{\partial x}(\bar x,0)(x_b-\bar x)+\eta \frac 12 \frac{\partial ^2v}{\partial y^2}(\bar x,0)y_b^2+O(|(x_b-\bar x)|^2)+o( y_b^2)\]
(where we used that $\frac{\partial v}{\partial y}(\bar x,0)=0$ and $\frac{\partial^2 v}{\partial xy}(\bar x,0)=0$
)
so that
\[\eta v(x_b,y_b)-1=\eta f(\bar x)-1+\eta f'(\bar x)(x_b-\bar x)
+O(y_b^2,|(x_b-\bar x)|^2)
\]
and
\[\frac 12 \log \left(1-\frac{4y_b^2}{(1+x_b^2+y_b^2)^2}\right)=-\frac{2y_b^2}{(1+\bar x^2)^2}+O\left(y_b^4,y_b^2|x_b-\bar x|\right).\]
Then \eqref{f56} becomes 
\[\frac{2y_b^2}{(1+\bar x^2)^2}+O\left(y_b^4,y_b^2|x_b-\bar x|\right)=b\left[1-\eta f(\bar x )-\eta f'(\bar x)(x_b-\bar x)
+O(y_b^2,(x_b-\bar x)^2)\right].\]
This imples that the ratio $\frac {y_b^2}b$ is bounded, and 
\beq\label{f66}
\frac{y_b^2}b=\frac 12(1+\bar x^2)^2\left[1-\eta f(\bar x )-\eta f'(\bar x)(x_b-\bar x)\right]\left(1
+O(y_b^2,|x_b-\bar x|^2,\frac{y_b^2}b|x_b-\bar x|)\right)\eeq
and \eqref{ma6} follows.
\end{proof}
\begin{remark}
As consequence of the previous lemma,  when $\bar x\neq \bar x_\eta$ then $y_b^2\sim b$ since $1-\eta f(\bar x)\neq0$. 
If $\bar x=\bar x_\eta$ instead,  then $y_b^2\sim b(x_b-\bar x_\eta)$ since $1-\eta f(\bar x_\eta)=0$ and $f'(\bar x_\eta)\neq 0$. 
\end{remark}

Next aim is to prove that the superlevel set
\beq\label{omega-b}
\omega_b:=\{(x,y)\in \Omega_b: u_b(x,y)>b\}
\eeq
has $n$ different components. We start with a first property,
\begin{lemma}\label{l4b}
Let $x_1,..,x_n$ be the minimum points of the polynomial $f(x)$ in $(-a_n,a_n)$. For every $b>0$, the points $(x_i,0)\in \omega_b$ for $i=1,..,n$.
\end{lemma}
\begin{proof}
We compute
\begin{equation}
u_b(x_i,0)=b-\eta bf(x_i)>b
\end{equation}
since $f(x_i)<0$ for $i=1,..,n$, by the definition of $f$.
\end{proof}
Next we have
\begin{lemma}\label{l5b}
Let $\tilde x_1,..,\tilde x_{n-1}$ be the maximum points of the polynomial $f(x)$ in $(-a_n,a_n)$.
There exists $b_1>0$ such that the superlevel set $\omega_b$ does not intersect the lines $x=\tilde x_i$ for $i=1,..,n-1$ for every $0<b<b_1$. 
\end{lemma}
\begin{proof}
We compute 
\[
u_b(\tilde x_i,y)\leq b\left(1-\eta v(\tilde x_i,y)\right).
\]
By Lemma \ref{ma11} if $(\tilde x_i,y)\in \Omega_b$ then $y\to 0$ as $b\to 0$.  Then $v(\tilde x_i,y)=v(\tilde x_i,0)+O(|y|^2)=f(\tilde x_i)+O(b)$ since $v$ is even in $y$ and $y_b=O(\sqrt b)$,
 and so
\[u_b(\tilde x_i,y)\leq b\left(1-\eta f(\tilde x_i)\right)(1+o(1))<b\left(1-\frac \eta 2f(\tilde x_i)\right)<0\]
if $b$ is small enough, by the definition of $\eta$. This gives the claim.
\end{proof}
\begin{corollary}\label{maxc}
The set $\omega_{b}$ admits at least $n$ connected components if $b$ is small enough.
\end{corollary}
\begin{proof}
By Lemma \ref{l4b} we have that the $n$ points $(x_i,0)$ (for $i=1,..,n$),  belong to $\omega_{b}$. On the other hand they belong to different connected components since the straight line $x=\tilde x_i$ does not belong to $\omega_{b}$ for $i=1,..,n-1$.  This ends the proof.
\end{proof}

In next proposition we prove the the domain $\Omega_b$ is starshaped with respect to the origin for $b$ small enough. The star-shapedness   of $\Omega_b$  with respect to the origin in the usual (Eucliean) sense, is equivalent to the geodesic starshapedness of $\Omega_b$ as a spherical or hyperbolic domain. In fact, in the two models that we are considering, straight lines through the origin are geodesics.
 \begin{proposition}\label{maxp}
The set $\Omega_b$ is smooth and starshaped with respect to the origin if $b$ is small enough.
\end{proposition}
\begin{proof}
We want to show that
\beq \label{ma53}
\langle\nu_E,(x,y)\rangle_E\ge\alpha _b>0\eeq
for some $\alpha_b>0$,  for every $(x,y)\in \partial \Omega_b$,
where $\langle\cdot,\cdot\rangle_E$ is the standard scalar product of $\mathbb R^2$
and $\nu_E(x,y)$ is the outher normal to $\partial \Omega_b$ at the point $(x,y)$,
when $b$ is small enough. In particular, by \eqref{3.1} we will show that 
\[\langle\frac{\nabla u_b}{|\nabla u_b|}(x,y),(x,y)\rangle_E\le-\alpha_b <0\]
for every $(x,y)\in \partial \Omega_b$,  when $b$ is small enough. This will show also that $\Omega_b$ is smooth.\\
Let us consider the spherical case first.  
Some computations give that
\[\frac{\partial u_b(x,y)}{\partial x}=\frac{8xy^2}{(1+x^2+y^2)\left[(1+x^2+y^2)^2-4y^2\right]}-b\eta \frac{\partial v (x,y)}{\partial x}
\]
\[\frac{\partial u_b(x,y)}{\partial y}=\frac {-4y(1+x^2-y^2)}{(1+x^2+y^2)\left[(1+x^2+y^2)^2-4y^2\right]}-b\eta \frac {\partial v(x,y)}{\partial y}.
\]
Let $(x_b,y_b)$ be a point on $\partial \Omega_b$. Recall that $y_b\to 0$ as $b\to 0$ (see Lemma \ref{ma11}) and we assume that 
$$x_b\to \bar x\in [-\bar x_\eta,\bar x_\eta].$$
Since $\Omega_b$ is symmetric with respect to $x=0$ and $y=0$ we consider hereafter only the case of $x\geq 0$ and $y\geq 0$.  Next we have
\[
\begin{split}
\frac{\partial v (x_b,y_b)}{\partial x}&=\frac{\partial v (\bar x,0)}{\partial x}+\frac{\partial^2 v (\bar x,0)}{\partial x^2}(x_b-\bar x)+\frac{\partial^2 v (\bar x,0)}{\partial x\partial y}y_b+O(y_b^2,(x_b-\bar x)^2)\\
 &=f'(\bar x)+O(y_b^2,|x_b-\bar x|)
\end{split}
\]
and
\[
\begin{split}
\frac{\partial v (x_b,y_b)}{\partial y}&=\frac{\partial v (\bar x,0)}{\partial y}+\frac{\partial^2 v (\bar x,0)}{\partial x\partial y}(x_b-\bar x)+\frac{\partial^2 v (\bar x,0)}{\partial y^2}y_b+O(y_b^2,(x_b-\bar x)^2)\\
&=\frac{\partial^2 v (\bar x,0)}{\partial y^2}y_b+O(y_b^2,(x_b-\bar x)^2)
\end{split}
\]
by the symmetry of $v(x,y)$. Moreover
\[
\frac{8x_by_b^2}{(1+x_b^2+y_b^2)\left[(1+x_b^2+y_b^2)^2-4y_b^2\right]}=\frac{8\bar xy_b^2}{(1+\bar x^2)^3}(1+O(|x_b-\bar x|,y_b^2))
\]
and
\[
\frac {-4y_b(1+x_b^2-y_b^2)}{(1+x_b^2+y_b^2)\left[(1+x_b^2+y_b^2)^2-4y_b^2\right]}=-4\frac{y_b}{(1+\bar x^2)^2}(1+O(|x_b-\bar x|,y_b^2)).
\]
Using \eqref{f66} we get the expansion of the first derivative of $u_b$ in $(x_b,y_b)$, 
{\small
\beq \label{derx}
\frac{\partial u_b(x_b,y_b)}{\partial x}=b\left(\frac{4\bar x}{(1+\bar x^2)}(1-\eta f(\bar x)
)-\eta f'(\bar x))
\right)(1+O(y_b^2,|x_b-\bar x|))
\eeq}
and
\beq\label{dery}
\begin{split}
\frac{\partial u_b(x_b,y_b)}{\partial y}&=\left[-4\frac{\sqrt{\frac b 2 (1-\eta f(\bar x)-\eta f'(\bar x)(x_b-\bar x))}}{(1+\bar x^2)}-\underbrace{b\eta \frac{\partial^2 v (\bar x,0)}{\partial y^2}y_b}_{=O( b y_b)}\right](1+O(|x_b-\bar x|,y_b^2))\\
&=-4\frac{\sqrt{\frac b 2\big(1-\eta f(\bar x)-\eta f'(\bar x)(x_b-\bar x)\big)}}{(1+\bar x^2)}(1+O(|x_b-\bar x|,\sqrt b y_b)).
\end{split}
\eeq
So we have
\begin{multline}\label{grad}
|\nabla u_b(x_b,y_b)|^2\\=b\left(8\frac{ (1-\eta f(\bar x)-\eta f'(\bar x)(x_b-\bar x))}{(1+\bar x^2)^2}+b\left(\frac{4\bar x}{(1+\bar x^2)}(1-\eta f(\bar x)
)-\eta f'(\bar x))\right)^2
\right)(1+o(1)).
\end{multline}
Now if $\bar x\neq \bar x_\eta$ we have that $ 1-\eta f(\bar x)\neq 0$ and  \eqref{grad} gives
\beq\label{ma52}
|\nabla u_b(x_b,y_b)|=2\sqrt 2\sqrt b\frac{\sqrt {1-\eta f(\bar x)}}{ (1+\bar x^2)}(1+o(1)).
\eeq
On the other hand if $\bar x=\bar x_\eta$ we get  $ 1-\eta f(\bar x_\eta)= 0$ and $f'(\bar x_\eta )\neq 0$. So \eqref{grad} becomes
\beq\label{ma54}
|\nabla u_b(x_b,y_b)|=\begin{cases}
2\sqrt 2\sqrt b\frac{\sqrt {\eta f'(\bar x_\eta)(\bar x_\eta-x_b)}}{ 1+\bar x_\eta^2}(1+o(1)) & \hbox{ when } \frac{|\bar x_\eta-x_b|}{b}\to \infty\\
 b 2\sqrt 2 \frac{\sqrt{\eta f'(\bar x_\eta)(\gamma+2\bar x_\eta^2 \eta f'(\bar x_\eta))
}}{ 1+\bar x_\eta^2}
(1+o(1))
& \hbox{ when } \frac{|\bar x_\eta-x_b|}{b}\to \gamma>0\\
b\frac{4\bar x_\eta}{1+\bar x_\eta^2}\eta f'(\bar x_\eta)(1+o(1))& \hbox{ when } \frac{|\bar x_\eta-x_b|}{b}\to 0.
\end{cases}
\eeq
Let us compute
{\small
\beq \label{ma55}
\begin{split}
\langle\frac{\nabla u_b(x,y)}{|\nabla u_b(x,y)|},(x,y)\rangle_E&=\frac b{|\nabla u_b(x,y)|}\left(\frac{4\bar x^2}{(1+\bar x^2)}(1-\eta f(\bar x)
)-\eta \bar x f'(\bar x)-2( 1-\eta f(\bar x))
\right)(1+O(y_b^2,|x_b-\bar x|))\\
&=\frac b {|\nabla u_b(x,y)|}\frac{2(1-\eta f(\bar x))(\bar x^2-1)-\eta (1+\bar x^2)\bar x f'(\bar x)}{(1+\bar x^2)}(1+o(1))
\end{split}
\eeq}
so that the proof of the starshapedness of $\Omega_b$ for $b$ small, is reduced to the proof of 
\beq\label{ma30}
-2(1-\eta f(\bar x))(1-\bar x^2)-\eta \bar xf'(\bar x)(1+\bar x^2)\leq -\alpha<0
\eeq
for every $\bar x\in[0,\bar x_\eta[$ for some positive number $\alpha$.

We consider two different cases: in the first one we will prove \eqref{ma30} for $0\le x < a_n$, in the second one for $x \ge a_n$.
\vskip0.2cm 
{\bf Case  $0\leq \bar x\le a_n$.}\\
By \eqref{eta} we know that
\[
\frac 1{\eta}>\sup_{(-a_n,a_n)}\left[f(x)-\frac {x(1+x^2)}{2(1-x^2)}f'(x)\right]
\]
it is then easy to see that
\[
\sup_{(0,a_n)}\Big[-2(1-x^2)+\eta\left( 2 f(x)(1-x^2) +xf'(x)(1+x^2)\right)\Big]<0.
\]
This proves \eqref{ma30} when $\bar x\in [0,a_n]$.
\vskip0.2cm
{\bf Case  $a_n\le \bar x\le \bar x_\eta$.}\\
We already know that $f( x)\le f( x_\eta)=\frac1\eta$ which gives that $1-\eta f(x)\ge0$ for $x\in [a_n,\bar x_\eta]$.  Moreover $f'(x)>0$ for $x>a_n$ which gives that 
\[\sup_{(a_n,\bar x_\eta)} 
-2(1-\eta f(x))(1-x^2)-\eta xf'(x)(1+x^2)\leq \sup_{(a_n,\bar x_\eta)}  -\eta xf'(x)(1+x^2)<0
\]
and this proves \eqref{ma30}.  So \eqref{ma53} follows by \eqref{ma54}, \eqref{ma55} and \eqref{ma30}. Although it is not essential in the proof, we write the expansion of $\alpha_b$ because it will be used in the computation of the curvature, 
\beq\label{ma57}
\alpha_b=
\begin{cases}
\frac {\alpha\sqrt b}{2\sqrt 2 \sqrt{1-\eta f(\bar x)}}
& \hbox{ when } x_b\to \bar x\neq \bar x_\eta\\
\frac{\alpha \sqrt b}{2\sqrt 2 \sqrt {\eta f'(\bar x_\eta)(\bar x_\eta-x_b)}} & \hbox{ when } x_b\to \bar x_\eta,  \frac{|\bar x_\eta-x_b|}{b}\to \infty\\
\frac \alpha{2 \sqrt 2 \sqrt{\eta f'(\bar x_\eta)(\gamma+2\bar x_\eta^2 \eta f'(\bar x_\eta))
}}
& \hbox{ when } \frac{|\bar x_\eta-x_b|}{b}\to \gamma>0\\
\frac{\alpha}{4\bar x_\eta \eta f'(\bar x_\eta)}& \hbox{ when } \frac{|\bar x_\eta-x_b|}{b}\to 0.
\end{cases}
\eeq

\

In the hyperbolic case,  reasoning as in \eqref{derx} and \eqref{dery},  we have
{\small
\[\frac{\partial u_b(x_b,y_b)}{\partial x}=b\left(\frac{-4\bar x}{1-\bar x^2}(1-\eta f(\bar x)
)-\eta f'(\bar x))
\right)(1+O(y_b^2,|x_b-\bar x|))
\]}
and
\[\frac{\partial u_b(x_b,y_b)}{\partial y}=\left[-4\frac{\sqrt{\frac b 2 (1-\eta f(\bar x)-\eta f'(\bar x)(x_b-\bar x))}}{1-\bar x^2}-\underbrace{b\eta \frac{\partial^2 v (\bar x,0)}{\partial y^2}y_b}_{=O(by_b)}\right](1+O(|x_b-\bar x|,y_b^2)).
\]
This gives the expansion of $|\nabla u_b(x_b,y_b)|$ 
\beq\label{ma54b}
|\nabla u_b(x_b,y_b)|=\begin{cases}
\sqrt b2\sqrt 2\frac{\sqrt {1-\eta f(\bar x)}}{ 1-\bar x^2}(1+o(1))& \hbox{ when } x_b\to \bar x\ne x_\eta\\
\sqrt b2\sqrt 2\frac{\sqrt {\eta f'(\bar x_\eta)(\bar x_\eta-x_b)}}{ 1-\bar x_\eta^2}(1+o(1)) & \hbox{ when } \frac{|\bar x_\eta-x_b|}{b}\to \infty\\
 b 2\sqrt 2 \frac{\sqrt{\eta f'(\bar x_\eta)(\gamma+2\bar x_\eta^2 \eta f'(\bar x_\eta))
}}{ 1-\bar x_\eta^2}
(1+o(1))
& \hbox{ when } \frac{|\bar x_\eta-x_b|}{b}\to \gamma>0\\
b\frac{4\bar x_\eta}{1-\bar x_\eta^2}\eta f'(\bar x_\eta)(1+o(1))& \hbox{ when } \frac{|\bar x_\eta-x_b|}{b}\to 0.
\end{cases}
\eeq
Then
{\small
\begin{equation}\label{ma73}
\begin{split}
\langle\nabla u_b(x,y),(x,y)\rangle_E&=b\frac{-2(1-\eta f(\bar x))(\bar x^2+1)-\eta (1-\bar x^2)\bar x f'(\bar x)}{1-\bar x^2}+o(b)
\end{split}
\end{equation}}
and letting $-\alpha=\sup_{(0,a_n)}-2(1-\eta f(\bar x))(\bar x^2+1)-\eta (1-\bar x^2)\bar x f'(\bar x)<0$ we get, for some positive constant $\bar\alpha>0$,
\begin{equation}
\langle\nu_E,(x,y)\rangle_E>\bar\alpha\sqrt b+o(\sqrt b)\quad\hbox{if  }x_b\to \bar x\ne x_\eta
\end{equation}
and, using \eqref{ma57}
\beq\label{ma57b}
\langle\nu_E,(x,y)\rangle_E=\alpha_b\ge0\quad\hbox{if  }x_b\to x_\eta.
\eeq
\end{proof}
We conclude this section showing that \eqref{cla} and \eqref{cla1} hold. It is convenient to look at the {\it stereographic projection} of a star-shaped domain with respect to the south pole $-P$, and translate condition \eqref{assumption} into a condition for the image domain in $\mathbb R^2$.

Note that since $\Omega$ is star-shaped with respect to $P$, $-P\notin\Omega$. We consider the stereographic projection $\Psi:\mathbb S^2\to\mathbb R^2$ from the south pole $-P$, which maps a point $(x,y,z)\in\mathbb S^2\subset\mathbb R^3$ to the point $(X,Y)=\left(\frac{x}{1+z},\frac{y}{1+z}\right)$ in $\mathbb R^2$. Let $\Omega_E:=\Psi(\Omega)\subset\mathbb R^2$. Now, $\Omega$ is uniformly star-shaped with respect to $P$, if and only if $\Omega_E$ is uniformly star-shaped with respect to $0$.

Let $\kappa_E$ be the geodesic curvature of $\partial\Omega_E$ for the Euclidean metric with respect to the unit outer normal $\nu_E$. Since the Euclidean metric is conformal to the metric of $\mathbb S^2$ in stereographic projection, namely, $\frac{4}{(1+X^2+Y^2)^2}(dX^2+dY^2)$, we have that
$$
\kappa=\frac{1+X^2+Y^2}{2}\kappa_E-\langle\nu_E,(X,Y)\rangle_E,
$$
where $\langle\cdot,\cdot\rangle_E$ is the standard scalar product of $\mathbb R^2$. Using this, we get that
\begin{equation}\label{assumption_stereo}
\cos(\theta)\kappa+\sin(\theta)\langle\nu,\vec{e_{\theta}}\rangle=\frac{1-X^2-Y^2}{2}\kappa_E+\langle\nu_E,(X,Y)\rangle_E.
\end{equation}
\begin{lemma}\label{sharp}
For $b$ small enough we have that
\begin{equation}\label{m91-f}
\begin{split}
&\lim\limits_{b\to0}(\cos(\theta)k_{\Omega_b}+\sin(\theta)\langle\nu,\vec{e_{\theta}}\rangle)\ge0\,\ \ \hbox{ on }\partial\Omega_b,\subset \mathbb S^2\\
&\lim\limits_{b\to0}(\cosh(r)k_{\Omega_b}-\sinh(r)\langle\nu,\vec{e_r}\rangle)\ge0\,\ \ \hbox{ on }\partial\Omega_b\subset \mathbb H^2,
\end{split}
\end{equation}
where $\theta={\rm dist}(\cdot,0)$ is the geodesic distance from the origin (i.e., from the north pole) in spherical coordinates and $r={\rm dist}(\cdot,0)$ is the geodesic distance from the origin in $ \mathbb H^2$.
\end{lemma}
\begin{proof}
In the proof of Proposition \ref{maxp} we have shown that $\langle\nu_E,(x,y)\rangle_E\ge\alpha _b>0$, see \eqref{ma53}. So it will be enough to show that $\lim_{b\to 0} \cos(\theta)k_{\Om_b}\geq 0$ for every $(x,y)\in \partial \Omega_b$.

{\bf Case 1, $\Omega_b\subset \mathbb S^2$}

Using stereographic coordinates we have,
\begin{equation}\label{mcur0}
\begin{split}
&\cos(\theta)k_{\Om_b}+\sin(\theta)\langle\nu,\vec{e_{\theta}}\rangle=\frac{1-x^2-y^2}2k_{\Om_b,E}+\langle\nu_E,(x,y)\rangle_E
\end{split}
\end{equation}
where $k_{\Om_b,E}$ is the euclidean curvature of $\partial \Omega_b$ and $\langle\nu_E,(x,y)\rangle_E$ has been computed in the proof of Proposition \ref{maxp}, see \eqref{ma53} with $\alpha_b$ as in \eqref{ma57}. 
Recalling that 
\[k_{\Om_b,E}=-\frac{\frac{\partial^2 u_b}{\partial x^2}\left(\frac{\partial u_b}{\partial y}\right)^2-2\frac{\partial^2 u_b}{\partial x\partial y}\frac{\partial u_b}{\partial x}\frac{\partial u_b}{\partial y}+\frac{\partial^2 u_b}{\partial y^2}\left(\frac{\partial u_b}{\partial x}\right)^2}{|\nabla u_b|^3}
\]
we expand the second derivatives of $u_b$ by using \eqref{ma6}. Let $(x_b,y_b)\in \Omega_b$ such that $x_b\to \bar x\in[-\bar x_\eta,-\bar x_\eta]$ as $b\to 0$.  As in the previous lemma we only consider the case of $x\geq 0$ and $y\geq 0$. 
Since $y_b\to 0$ as $b\to 0$ a straightforward computation gives
\[\begin{split}
\frac{\partial^2 \psi_b}{\partial x^2}(x_b,y_b)&=8y_b^2\frac{1-5\bar x^2 }{(1+\bar x^2)^4}(1+O(y_b^2,|x_b-\bar x|))\\
&=4b\Big(1-\eta f(\bar x)-\eta f'(\bar x)(x_b-\bar x)\Big)\frac{1-5\bar x^2 }{(1+\bar x^2)^2}(1+O(y_b^2,|x_b-\bar x|))
\end{split}
\]
\[\begin{split}
\frac{\partial^2 \psi_b}{\partial x\partial y}(x_b,y_b)&=16y_b
\frac{\bar x}{(1+\bar x^2)^3}(1+O(y_b^2,|x_b-\bar x|))\\
&=16 \sqrt{\frac b2 \Big(1-\eta f(\bar x)-\eta f'(\bar x)(x_b-\bar x)\Big)}\frac{\bar x }{(1+\bar x^2)^2}(1+O(y_b^2,|x_b-\bar x|))
\end{split}
\]
\[\begin{split}
\frac{\partial^2 \psi_b}{\partial y^2}(x_b,y_b)&=-\frac 4{(1+\bar x^2)^2}(1+O(y_b^2,|x_b-\bar x|)).
\end{split}
\]
Moreover, by the symmetry of $v(x,y)$,
\[
\begin{split}
\frac{\partial^2 v (x_b,y_b)}{\partial x^2}&=\frac{\partial^2 v (\bar x,0)}{\partial x^2}+O(y_b^2,|x_b-\bar x|)=f''(\bar x)+O(y_b^2,|x_b-\bar x|)\\
\frac{\partial^2 v (x_b,y_b)}{\partial x \partial y}&=\frac{\partial^2 v (\bar x,0)}{\partial x \partial y}+O(y_b,|x_b-\bar x|^2)=O(y_b,|x_b-\bar x|^2)\\
\frac{\partial^2 v (x_b,y_b)}{\partial y^2}&=\frac{\partial^2 v (\bar x,0)}{\partial y^2}+O(y_b^2,|x_b-\bar x|).\\
\end{split}
\]
Then
{\small
\[
\begin{split}
\frac{\partial^2 u_b}{\partial x^2}(x_b,y_b)&=b\left[
4\Big(1-\eta f(\bar x)-\eta f'(\bar x)(x_b-\bar x)\Big)\frac{1-5\bar x^2 }{(1+\bar x^2)^2}-\eta f''(\bar x)
\right](1+O(y_b^2,|x_b-\bar x|))\\
\frac{\partial^2 u_b}{\partial x\partial y}(x_b,y_b)&=\sqrt b\left[
16 \sqrt{\frac 12 \Big(1-\eta f(\bar x)-\eta f'(\bar x)(x_b-\bar x)\Big)}\frac{\bar x }{(1+\bar x^2)^2}
\right](1+O(\sqrt b y_b,|x_b-\bar x|))\\
\frac{\partial^2 u_b}{\partial y^2}(x_b,y_b)&=-\frac 4{(1+\bar x^2)^2}(1+O(b,|x_b-\bar x|)).
\end{split}
\]}
Then we have, using \eqref{derx} and \eqref{dery}
\beq\label{st}
\begin{split}
&\frac{\partial^2 u_b}{\partial x^2}\left(\frac{\partial u_b}{\partial y}\right)^2-2\frac{\partial^2 u_b}{\partial x\partial y}\frac{\partial u_b}{\partial x}\frac{\partial u_b}{\partial y}+\frac{\partial^2 u_b}{\partial y^2}\left(\frac{\partial u_b}{\partial x}\right)^2=\\
&8b^2\left[
4\Big(1-\eta f(\bar x)-\eta f'(\bar x)(x_b-\bar x)\Big)\frac{1-5\bar x^2 }{(1+\bar x^2)^2}-\eta f''(\bar x)
\right]
\frac{1-\eta f(\bar x)-\eta f'(\bar x)(x_b-\bar x)}{(1+\bar x^2)^2}
\big(1+o(1)\big)\\
&+8b^2\left[
16 \sqrt{\frac 12 \Big(1-\eta f(\bar x)-\eta f'(\bar x)(x_b-\bar x)\Big)}\frac{\bar x }{(1+\bar x^2)^2}
\right]\left(\frac{4\bar x(1-\eta f(\bar x))
}{1+\bar x^2}-\eta f'(\bar x))
\right)\\
&
\cdot
\frac{\sqrt{\frac 1 2 (1-\eta f(\bar x)-\eta f'(\bar x)(x_b-\bar x))}}{1+\bar x^2}\big(1+o(1)\big)-b^2\frac 4{(1+\bar x^2)^2}\left(\frac{4\bar x(1-\eta f(\bar x))
}{1+\bar x^2}-\eta f'(\bar x))\right)^2\big(1+o(1)\big)
\end{split}
\eeq
So when $\bar x\neq \bar x_\eta$ we have that $ 1-\eta f(\bar x)\neq 0$ and  
\[
|\nabla u_b(x_b,y_b)|=2\sqrt 2\sqrt b\frac{\sqrt {1-\eta f(\bar x)}}{1+\bar x^2}(1+o(1))
\]
so that 
\begin{equation}\label{nuova-fra1}
k_{\Om_b,E}=-b^2 \frac{G(\bar x)}{b^\frac 32 \left(2\sqrt 2\frac{\sqrt {1-\eta f(\bar x)}}{ (1+\bar x^2)}\right)^3}(1+o(1))=O(\sqrt b)
\end{equation}
where 
\begin{equation}\label{nuova-fra2}
\begin{split}
G(\bar x)&=\left[
32\Big(1-\eta f(\bar x)\Big)\frac{1-5\bar x^2 }{(1+\bar x^2)^2}-8\eta f''(\bar x)
\right]
\frac{1-\eta f(\bar x) }{(1+\bar x^2)^2}\\
&+64\Big(1-\eta f(\bar x)\Big)\frac{\bar x}{(1+\bar x^2)^3}\left(\frac{4\bar x(1-\eta f(\bar x))
}{1+\bar x^2}-\eta f'(\bar x))\right)\\
&- \frac 4{(1+\bar x^2)^2}\left(\frac{4\bar x}{(1+\bar x^2)}(1-\eta f(\bar x)
)-\eta f'(\bar x))
\right)^2.
\end{split}
\end{equation}
So we get
\begin{equation}
\frac{1-x^2-y^2}2k_{\Om_b,E}+\langle\nu_E,(x,y)\rangle_E=O(\sqrt b),
\end{equation}
and so by \eqref{mcur0} we get \eqref{m91-f}.
\vskip0.2cm
When $\bar x=\bar x_\eta$ instead,  $ 1-\eta f(\bar x_\eta)= 0$ and $f'(\bar x_\eta )\neq 0$ and \eqref{ma54} holds. So \eqref{st} becomes
\beq\label{st2}
\begin{split}
&\frac{\partial^2 u_b}{\partial x^2}\left(\frac{\partial u_b}{\partial y}\right)^2-2\frac{\partial^2 u_b}{\partial x\partial y}\frac{\partial u_b}{\partial x}\frac{\partial u_b}{\partial y}+\frac{\partial^2 u_b}{\partial y^2}\left(\frac{\partial u_b}{\partial x}\right)^2=\\
&8b^2\left[
4\eta f'(\bar x)(\bar x-x_b)\frac{1-5 \bar x^2 }{(1+\bar x^2)^2}-\eta f''(\bar x)
\right]\frac{\eta f'(\bar x)(\bar x-x_b) }{(1+\bar x^2)^2}
\big(1+o(1)\big)\\
&-64b^2\frac{\bar x\eta^2f'(\bar x)^2(\bar x-x_b)}{(1+\bar x^2)^3}\big(1+o(1)\big)-4b^2\frac{\eta^2(f'(\bar x))^2}{(1+\bar x^2)^2}\big(1+o(1)\big)=\\
&-4b^2\frac{\eta^2(f'(\bar x))^2}{(1+\bar x^2)^2}\big(1+o(1)\big)<0.
\end{split}
\eeq
This gives that $k_{\Om_b,E}>0$ if $\bar x-x_b\to0$ and by \eqref{ma53} we get
\begin{equation}\label{fra-nuova}
\frac{1-x^2-y^2}2k_{\Om_b,E}+\langle\nu_E,(x,y)\rangle_E>0
\end{equation}
which proves \eqref{m91-f} when $\partial\Omega_b\subset \mathbb S^2$.
\vskip0.2cm
{\bf Case 2, $\Omega_b\subset \mathbb H^2$}

Using Euclidian coordinates we have 
\begin{equation}\label{ti}
\cos(\theta)k_{\Om_b}+\sin(\theta)\langle\nu,\vec{e_{\theta}}\rangle=
\frac{1+x^2+y^2}{2}\kappa_E-\langle\nu_E,(x,y)\rangle_E.
\end{equation}
Here the computations are very similar but the contribution of $\langle\nu_E,(x,y)\rangle_E$ has a ``bad'' sign. We have that
{\small
\[
\begin{split}
\frac{\partial^2 u_b}{\partial x^2}(x_b,y_b)&=-b\left[
4\Big(1-\eta f(\bar x)-\eta f'(\bar x)(x_b-\bar x)\Big)\frac{1+5\bar x^2}{(1-\bar x^2)^4}+\eta f''(\bar x)
\right](1+O(y_b^2,|x_b-\bar x|))\\
\frac{\partial^2 u_b}{\partial x\partial y}(x_b,y_b)&=-\sqrt b\left[
16 \sqrt{\frac 12 \Big(1-\eta f(\bar x)-\eta f'(\bar x)(x_b-\bar x)\Big)}\frac{\bar x }{(1-\bar x^2)^2}
\right](1+O(\sqrt b y_b,|x_b-\bar x|))\\
\frac{\partial^2 u_b}{\partial y^2}(x_b,y_b)&=-\frac 4{(1-\bar x^2)^2}(1+O(b,|x_b-\bar x|)).
\end{split}
\]}
and then
\beq\label{st3}
\begin{split}
&\frac{\partial^2 u_b}{\partial x^2}\left(\frac{\partial u_b}{\partial y}\right)^2-2\frac{\partial^2 u_b}{\partial x\partial y}\frac{\partial u_b}{\partial x}\frac{\partial u_b}{\partial y}+\frac{\partial^2 u_b}{\partial y^2}\left(\frac{\partial u_b}{\partial x}\right)^2=\\
&-8b^2\left[
4\Big(1-\eta f(\bar x)-\eta f'(\bar x)(x_b-\bar x)\Big)\frac{1+5\bar x^2}{(1-\bar x^2)^4}+\eta f''(\bar x)
\right]
\frac{1-\eta f(\bar x)-\eta f'(\bar x)(x_b-\bar x)}{(1-\bar x^2)^2}
\big(1+o(1)\big)\\
&+8b^2\left[
16 \sqrt{\frac 12 \Big(1-\eta f(\bar x)-\eta f'(\bar x)(x_b-\bar x)\Big)}\frac{\bar x }{(1-\bar x^2)^2}
\right]\left(\frac{4\bar x(1-\eta f(\bar x))
}{1-\bar x^2}+\eta f'(\bar x))
\right)\\
&
\cdot
\frac{\sqrt{\frac 1 2 (1-\eta f(\bar x)-\eta f'(\bar x)(x_b-\bar x))}}{1-\bar x^2}\big(1+o(1)\big)-b^2\frac 4{(1-\bar x^2)^2}\left(\frac{4\bar x(1-\eta f(\bar x))
}{1-\bar x^2}+\eta f'(\bar x))\right)^2\big(1+o(1)\big)
\end{split}
\eeq
As in the previous case, when $\bar x\neq \bar x_\eta$ we have that $ 1-\eta f(\bar x)\neq 0$ and  
\[
|\nabla u_b(x_b,y_b)|=2\sqrt 2\sqrt b\frac{\sqrt {1-\eta f(\bar x)}}{1-\bar x^2}(1+o(1))
\]
so that again $k_{\Om_b,E}=O(b^\frac12)\to 0$. Next observe that by \eqref{ma54b} and \eqref{ma73} we deduce that $\langle\nu_E,(x,y)\rangle_E=O(b^\frac12)$ which gives
\begin{equation}
\cos(\theta)k_{\Om_b}+\sin(\theta)\langle\nu,\vec{e_{\theta}}\rangle=
\frac{1+x^2+y^2}{2}\kappa_E-\langle\nu_E,(x,y)\rangle_E=O(b^\frac12)
\end{equation}
and then  \eqref{m91-f} holds.
\vskip0.2cm
When $\bar x=\bar x_\eta$ instead,  $ 1-\eta f(\bar x_\eta)= 0$ and $f'(\bar x_\eta )\neq 0$ and again
\beq\label{st4}
\begin{split}
&\frac{\partial^2 u_b}{\partial x^2}\left(\frac{\partial u_b}{\partial y}\right)^2-2\frac{\partial^2 u_b}{\partial x\partial y}\frac{\partial u_b}{\partial x}\frac{\partial u_b}{\partial y}+\frac{\partial^2 u_b}{\partial y^2}\left(\frac{\partial u_b}{\partial x}\right)^2=-4b^2\frac{\eta^2f'(\bar x)^2}{(1-\bar x^2)^2}\big(1+o(1)\big)
\end{split}
\eeq
On the other hand by \eqref{ma54b} we have for some positive constants $C_1,C_2$
\beq
k_{\Om_b,E}=\begin{cases}
\big(C_1+o(1)\Big)\frac{\sqrt b}{|x_b-x_\eta|^\frac32}
& \hbox{ when } \frac{|\bar x_\eta-x_b|}{b}\to \infty\\
\frac{C_2+o(1)}b
& \hbox{ when } \frac{|\bar x_\eta-x_b|}{b}\to \gamma\ge0
\end{cases}
\eeq
and again since 
\[\langle\nu_E,(x,y)\rangle_E=\begin{cases}
O(\sqrt{\frac b{|x_b-\bar x_\eta|}}) & \hbox{ when } \frac{|\bar x_\eta-x_b|}{b}\to \infty\\\\
O(1) & \hbox{ when } \frac{|\bar x_\eta-x_b|}{b}\to \gamma\ge0
\end{cases}\]
this gives
\begin{equation}\label{nuova-fra3}
\cos(\theta)k_{\Om_b}+\sin(\theta)\langle\nu,\vec{e_{\theta}}\rangle=
\frac{1+x^2+y^2}{2}\kappa_E-\langle\nu_E,(x,y)\rangle_E=
\begin{cases}
>0 & \hbox{ when } \frac{|\bar x_\eta-x_b|}{b}\to \infty\\
+\infty & \hbox{ when } \frac{|\bar x_\eta-x_b|}{b}\to \gamma\ge0\end{cases}
\end{equation}
which proves \eqref{m91-f} for $\mathbb H^2$.

\end{proof}

We are now ready to conclude the proof of Theorem \ref{T1}.

\begin{proof}[Proof of Theorem \ref{T1}] Property $a)$ follows by Corollary \ref{maxc} because the existence of $n$ connected component for the superlevel $\omega_b$, defined in \eqref{omega-b},  implies the existence of at least $n$ maximum points for $u_b$. Property $b)$ is proved in Lemma \ref{l2b}, see \eqref{ma3}, $c)$ in Proposition \ref{maxp}.  Property $d)$ 
follows from in Lemma \ref{ma11} with $\bar x\neq \bar x_\eta$. Finally property $e)$ is proved in Lemma a \ref{sharp}.
\end{proof}
\begin{remark}
Here we make some comments to understand which part of the boundary of $\Omega_b$ does not satisfy the condition \eqref{assumption} and \eqref{assumption2} in the examples.  We consider first the case of the sphere $\mathbb S^2$.  By \eqref{fra-nuova}, if $(x_b,y_b)\in \partial \Omega_b$ is such that $\bar x_\eta-x_b\to0$ then condition \eqref{assumption} is satisfied.  For $\bar x\in (-\bar x_\eta,\bar x_\eta)$ we can consider the points $(\bar x, y_{b,\bar x})\in \partial \Omega_b$.  By \eqref{ma55} and \eqref{ma52} we have
\[\langle\nu_E,(x,y)\rangle_E=\frac {\sqrt b} {2\sqrt 2 \sqrt{ 1-\eta f(\bar x)}}\left(2(1-\eta f(\bar x))(1-\bar x^2)+\eta (1+\bar x^2)\bar x f'(\bar x)\right)
(1+o(1)),
\]
and, by using \eqref{nuova-fra1}
\[\frac{1-x^2-y^2}2k_{\Om_b,E}=\frac {\sqrt b}2 (1-\bar x^2)
\frac{G(\bar x)}{ \left(2\sqrt 2\frac{\sqrt {1-\eta f(\bar x)}}{ (1+\bar x^2)}\right)^3}(1+o(1))
\]
where $G(\bar x)$ is as in \eqref{nuova-fra2}. Then condition \eqref{assumption} becomes 
\[
\begin{split}
&\frac{1-x^2-y^2}2k_{\Om_b,E} +\langle\nu_E,(x,y)\rangle_E=\\
&\frac {(1-\bar x^2)(1+\bar x^2)^3\sqrt b}{ \left(2\sqrt 2\sqrt {1-\eta f(\bar x)}\right)^3}\left(\underbrace{16\frac{(1-\eta f(\bar x))^2}{(1+\bar x^2)^3}-\frac 12 G(\bar x)+8\eta \bar x f'(\bar x)\frac{(1-\eta f(\bar x))^2}{(1-\bar x^2)(1+\bar x^2)^2}}_{:=H(\bar x)}
\right)(1+o(1))
\end{split}
\]
This shows that condition \eqref{assumption} is not satisfied when $H(\bar x)<0$. 
In particular it can be easily seen that in the maximum points $\tilde x_i$ of $f(x)$ one has
\[H(\tilde x_i)=8 \frac{(1-\eta f(\tilde x_i))^2}{(1+ \tilde x_i^2)^4}(-3\tilde x_i^2)+4\eta f''(\tilde x_i)
\frac{1-\eta f( \tilde x_i) }{(1+\tilde x_i^2)^2}<0.
\]
In the case of the hyperbolic space $\mathbb H^2$ the computations are similar.  By \eqref{nuova-fra3} we have that in the points $(x_b,y_b)\in \partial \Omega_b$ such that $\bar x_\eta-x_b\to0$
condition \eqref{assumption2} is satisfied.  For $\bar x\in (-\bar x_\eta,\bar x_\eta)$ we can consider, as before,   the points $(\bar x, y_{b,\bar x})\in \partial \Omega_b$ and, by \eqref{ma73} and \eqref{ma54b} we have 
\[\langle\nu_E,(x,y)\rangle_E=\frac {\sqrt b} {2\sqrt 2 \sqrt{ 1-\eta f(\bar x)}}\left(2(1-\eta f(\bar x))(1+\bar x^2)+\eta (1-\bar x^2)\bar x f'(\bar x)\right)
(1+o(1)).
\]
Using \eqref{st3} we get
\[\frac{1+x^2+y^2}2k_{\Om_b,E}=\frac {\sqrt b}2 (1+\bar x^2)
\frac{\bar G(\bar x)}{ \left(2\sqrt 2\frac{\sqrt {1-\eta f(\bar x)}}{ (1-\bar x^2)}\right)^3}(1+o(1))
\]
where 
\[
\begin{split}
\bar G(\bar x)&=\left[
32\Big(1-\eta f(\bar x)\Big)\frac{1+5\bar x^2 }{(1+\bar x^2)^2}+8\eta f''(\bar x)
\right]
\frac{1-\eta f(\bar x) }{(1-\bar x^2)^2}\\
&-64\Big(1-\eta f(\bar x)\Big)\frac{\bar x}{(1-\bar x^2)^3}\left(\frac{4\bar x(1-\eta f(\bar x))
}{1-\bar x^2}+\eta f'(\bar x))\right)\\
&+ \frac 4{(1-\bar x^2)^2}\left(\frac{4\bar x}{(1-\bar x^2)}(1-\eta f(\bar x)
)+\eta f'(\bar x))
\right)^2.
\end{split}
\]
Then the condion \eqref{assumption2} becomes
\[
\begin{split}
&\frac{1+x^2+y^2}2k_{\Om_b,E} -\langle\nu_E,(x,y)\rangle_E=\\
&\frac {(1+\bar x^2)(1-\bar x^2)^3\sqrt b}{ \left(2\sqrt 2\sqrt {1-\eta f(\bar x)}\right)^3}\left(\underbrace{-16\frac{(1-\eta f(\bar x))^2}{(1-\bar x^2)^3}+\frac 12 \bar G(\bar x)
-8\eta \bar x f'(\bar x)\frac{(1-\eta f(\bar x))^2}{(1+\bar x^2)(1-\bar x^2)^2}
}_{:=\bar H(\bar x)}\right)(1+o(1)).
\end{split}
\]
Then the condition \eqref{assumption2} is not satisfied when $\bar H(\bar x)<0$.
In a maximum point $\tilde x_i$ of $f(x)$, as before,  we have 
\[\bar H(\tilde x_i)=8 \frac{(1-\eta f(\tilde x_i))^2}{(1- \tilde x_i^2)^4}(-3\tilde x_i^2)+4\eta f''(\tilde x_i)
\frac{1-\eta f( \tilde x_i) }{(1-\tilde x_i^2)^2}<0.
\]
\end{remark}
\section{Some examples}\label{examples}

\subsection{Convex spherical domains} 
Every smooth convex spherical domain $\Omega$ is uniformly star-shaped with respect to any of its points. Moreover, one can choose $P\in\Omega$ such that $\Omega$ is contained in the open hemisphere centered at $P$ and for such a choice of $P$, condition \eqref{assumption} is satisfied. This extends the result of \cite{GP1}, where Theorem \ref{main} was proved only for uniformly convex domains with diameter smaller than $\frac{\pi}{2}$. 

\subsection{Complements of convex spherical domains} Note that also the complements of spherical convex sets are uniformly star-shaped with respect to some points and satisfy \eqref{assumption}, hence Theorem \ref{main} applies. For example geodesic disks of radius greater than $\frac{\pi}{2}$ are not convex (their complements are convex). This fact is somewhat surprising when compared with the Euclidean case. Obviously, in our case the topology of $\mathbb S^2$ plays a crucial role. Moreover observe that the torsion function and the first Dirichlet eigenfunction, which are explicit, have indeed exactly one non-degenerate critical point. This phenomenon is more general, and the uniqueness of the critical point holds for any semi-stable solution and for complements of arbitrary convex sets.

\subsection{Other examples of non-convex, uniformly star-shaped spherical domains satisfying \eqref{assumption} obtained via stereographic projection}\ 

In order to  obtain other examples of spherical domains for which Theorem \ref{main} applies, it is convenient to look at the {\it stereographic projection} used in the previous section.

We use \eqref{assumption_stereo} to produce examples of uniformly star-shaped, non-convex domains on $\mathbb S^2$ for which \eqref{assumption} holds. To do so, it is sufficient to find uniformly star-shaped domains of $\mathbb R^2$ for which the right-hand side of \eqref{assumption_stereo} is positive. Note that in many situations, proving that \eqref{assumption_stereo} is positive is easier to check.

A first example is the following. Let $\Omega_p:=\Psi^{-1}(Q_p)$, where $Q_p=\{(X,Y)\in\mathbb R^2:|X|^p+|Y|^p<2^{-p}\}$, with $p\in\mathbb N$ sufficiently large. It is easy to check that $Q_p$ is uniformly star-shaped with respect to $0$ and is such that the right-hand side of \eqref{assumption_stereo} is positive, therefore $\Omega_p$ is uniformly starshaped with respect to the north pole, satisfies \eqref{assumption}, and if $p$ is large enough, it is not convex. Note that $Q_p$ is convex in $\mathbb R^2$.

Analogously, one can produce examples of uniformly star-shaped domains of $\mathbb S^2$ satisfying \eqref{assumption} starting from uniformly star-shaped domains $D$ of $\mathbb R^2$ which are {\it not convex in $\mathbb R^2$} (which implies that $\Psi^{-1}(D)$ is not convex in $\mathbb S^2$). For example, we can start from a uniformly star-shaped domain $D_{a,b,k}\subset\mathbb R^2$ whose boundary is of the form 
$$
\partial D_{a,b,k}=\{(X,Y)=(a+b\sin(2\pi k s))(\cos(s),\sin(s)):s\in(0,2\pi)\},
$$ 
where $a>0$, $b\in\mathbb R$ and $k\in\mathbb N$ are chosen such that the right-hand side of \eqref{assumption_stereo} is positive. Examples of domains $\Psi^{-1}(Q_p),\Psi^{-1}(D_{a,b,k})$ are depicted in Figure \ref{fig1}.

\begin{figure}
\includegraphics[width=\textwidth]{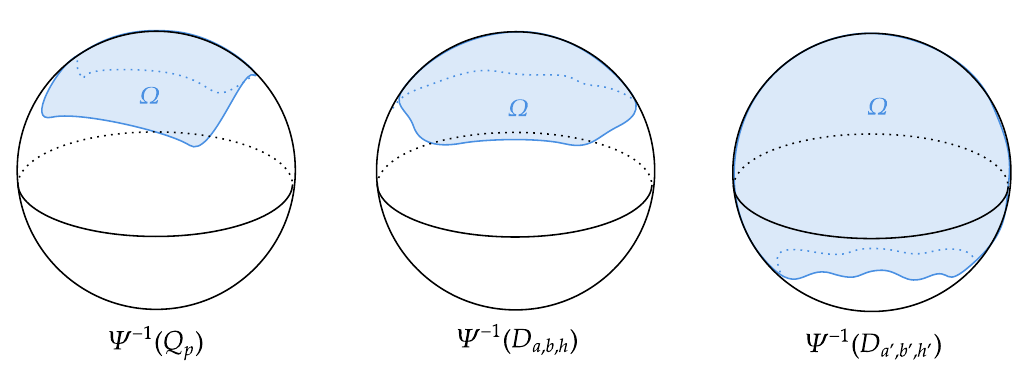}
\caption{Some examples of non-convex domains with a unique critical point.}
\label{fig1}
\end{figure}

\subsection{On the location of the critical point for spherical convex sets}\label{sub:location} 
If $\Omega$ is a smooth convex spherical domain, then it is uniformly star-shaped with respect to any of its points. Let $P_0\in\Omega$ be such that $\Omega$ is contained in the open hemisphere centered at $P_0$, and let  $G_\Omega$ be the connected component containing $P_0$ of the open set of points $P\in\Omega$ such that \eqref{assumption} holds. From Subsection \ref{nondegenerate} we deduce that the critical points do not belong to any equator with north pole $P$, for all $P\in G_\Omega$. This implies that, defining $\omega=\Omega\cap\bigcap_{P\in G_\Omega}B(P,\pi/2)$, then the critical point must be located in $\omega$. See Figure \ref{fig2} for a graphic representation. We also remark that using this information in \cite{GP1} would allow to drop the hypothesis on the diameter for uniformly convex, spherical domains. In fact, if the diameter of $\Omega$ is smaller than $\frac{\pi}{2}$, then $\omega=\Omega$, hence the result of Subsection \ref{nondegenerate} does not restrict the region in $\Omega$ where the critical point can be located. In any case, Theorem \ref{main} provides a stronger statement.

\begin{figure}
\includegraphics[width=\textwidth]{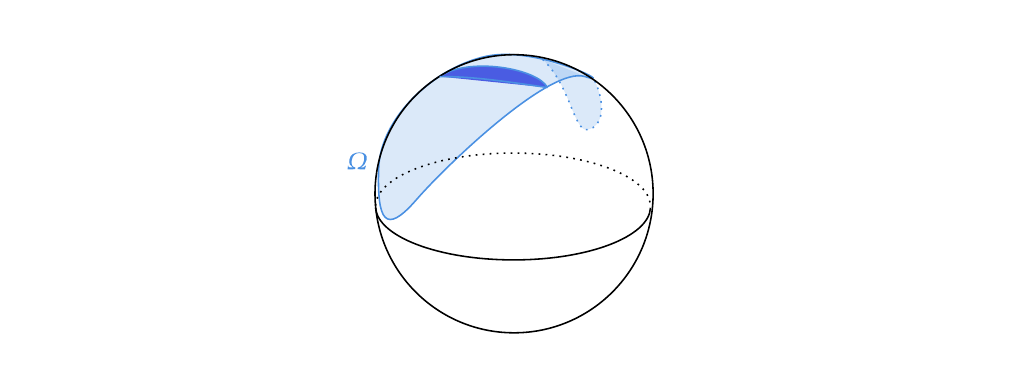}
\caption{The critical point of the convex set $\Omega$ must belong to the blue region.}
\label{fig2}
\end{figure}

\subsection{Non horoconvex domains}\label{hyp_ex}

We show here that the class of hyperbolic domains to which Theorem \ref{main2} applies is larger than horoconvex domains. 

Consider in the Poincaré disk model the ellipse $\Omega_{a,b}:=\{\frac{x^2}{a^2}+\frac{y^2}{b^2}<1\}$, where $0<b\leq a<1$. If $b<a^2$, $\Omega_ {a,b}$ is not horoconvex, since the (euclidean) curvature at the point $(0,b)$ is $\frac{b}{a^2}<1$. Clearly $\Omega_{a,b}$ is uniformly star-shaped with respect to the origin. We check that condition \eqref{assumption2} is satisfied for suitable choices of $a,b$ with $b<a^2$. It is more convenient to verify the equivalent assumption (see \eqref{equiv2})
$$
\frac{1+x^2+y^2}{2}\kappa_E-\langle\nu_E,(x,y)\rangle_E>0.
$$
Writing $x(\phi)=a\cos(\phi)$, $y(\phi)=b\sin(\phi)$, $\phi\in[0,2\pi]$, we compute 
\begin{equation}\label{condE}
\frac{1+x^2+y^2}{2}\kappa_E-\langle\nu_E,(x,y)\rangle_E\\
=\frac{ab(1-2a^2+b^2+3(a^2-b^2)\cos^2(\phi))}{2(a^2\sin^2(\phi)+b^2\cos^2(\phi))^{3/2}}.
\end{equation}
A sufficient condition ensuring that the right-hand side of  \eqref{condE} is positive and $\Omega_{a,b}$ is not horoconvex is that $(a,b)\in (0,1)^2$ and $\sqrt{2a^2-1}<b<a^2$. Examples of convex domains satisfying \eqref{assumption2} which are not horoconvex are shown in Figure \ref{fig3}.

\begin{figure}
\includegraphics[width=\textwidth]{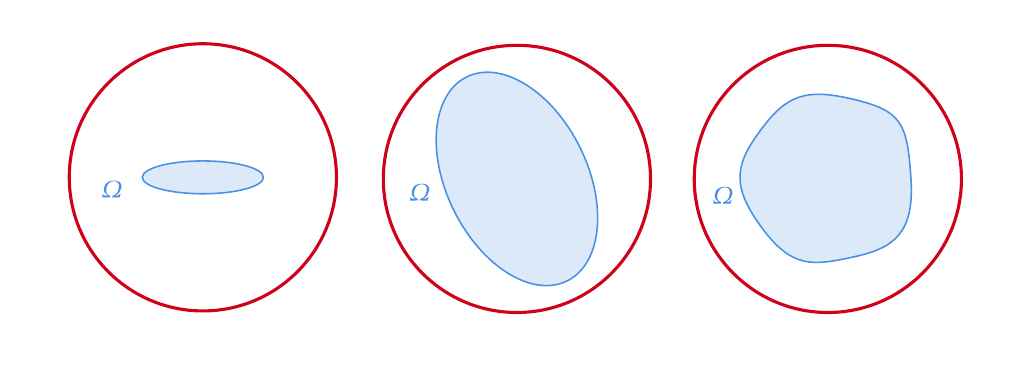}
\caption{Some examples of non-horoconvex domains with a unique critical point.}
\label{fig3}
\end{figure}


\bibliography{bibliography.bib}
\bibliographystyle{abbrv}
\end{document}